\documentclass[12 pt,reqno]{amsart}
\usepackage{amssymb,amsmath,amstext,amsgen,amsfonts,amsthm,verbatim,hyperref,fullpage,enumerate}
\usepackage[usenames,dvipsnames]{color}
\usepackage[nobysame,alphabetic]{amsrefs}

\def\Xint#1{\mathchoice
   {\XXint\displaystyle\textstyle{#1}}%
   {\XXint\textstyle\scriptstyle{#1}}%
   {\XXint\scriptstyle\scriptscriptstyle{#1}}%
   {\XXint\scriptscriptstyle\scriptscriptstyle{#1}}%
   \!\int}
\def\XXint#1#2#3{{\setbox0=\hbox{$#1{#2#3}{\int}$}
     \vcenter{\hbox{$#2#3$}}\kern-.5\wd0}}
\newcommand{\fint}{\Xint-}

\newcommand{\N}{\mathbb{N}}
\newcommand{\R}{\mathbb{R}}

\newcommand{\bP}{\mathbb{P}}
\newcommand{\cA}{\mathcal{A}}

\renewcommand{\epsilon}{\varepsilon}
\renewcommand{\phi}{\varphi}
\renewcommand{\hat}{\widehat}
\renewcommand{\tilde}{\widetilde}

\newcommand{\Hd}{\mathcal{H}}
\DeclareMathOperator{\diam}{diam}
\DeclareMathOperator{\dom}{dom}
\DeclareMathOperator{\im}{Image}
\DeclareMathOperator{\lip}{lip}
\DeclareMathOperator{\Lip}{Lip}

\DeclareMathOperator{\len}{len}
\DeclareMathOperator{\ms}{ms}
\DeclareMathOperator{\dist}{dist}
\DeclareMathOperator{\md}{md}
\DeclareMathOperator{\Tan}{Tan}
\DeclareMathOperator{\rad}{rad}

\newtheorem{theorem}{Theorem}[section]
\newtheorem{proposition}[theorem]{Proposition}
\newtheorem{lemma}[theorem]{Lemma}
\newtheorem{definition}[theorem]{Definition}
\newtheorem{corollary}[theorem]{Corollary}

\theoremstyle{remark}
\newtheorem{remark}[theorem]{Remark}

\begin{document}
\title{Differentiability and Poincar\'e-type inequalties in metric measure spaces}
\author{David Bate \and Sean Li}
\date{\today}
\address{Department of Mathematics, The University of Chicago, Chicago, IL 60637}
\email{bate@math.uchicago.edu}
\email{seanli@math.uchicago.edu}
\begin{abstract}
  We demonstrate the necessity of a Poincar\'e type inequality for those metric measure spaces that satisfy Cheeger's generalization of Rademacher's theorem for all Lipschitz functions taking values in a Banach space with the Radon-Nikodym property.  This is done by showing the existence of a rich structure of curve fragments that connect nearby points, similar in nature to Semmes's pencil of curves for the standard Poincar\'e inequality.  Using techniques similar to Cheeger-Kleiner \cite{cheeger-kleiner}, we show that our conditions are also sufficient.

  We also develop another characterization of \emph{RNP Lipschitz differentiability spaces} by connecting points by curves that form a rich structure of \emph{partial derivatives} that were first discussed in \cite{bate}.
\end{abstract}
\maketitle
\tableofcontents

\section{Introduction}
The last two decades have seen much activity in the development of first-order calculus in the setting of metric measure spaces.  One of the cornerstones of this field has been the theory of differentiation for functions on metric measure spaces as introduced by Cheeger in the seminal work \cite{cheeger} (see also \cites{gigli, kleiner-mackay}).  There, Cheeger says a function $f \colon X \to \R$ is differentiable with respect to another function $\phi \colon X \to \R^n$ at $x_0 \in X$ if there exists a $Df(x_0)\in L(\mathbb R^n,\mathbb R)$ such that the usual first order linear approximation holds:
\[f(x)-f(x_0) = Df(x_0)\cdot (\phi(x)-\phi(x_0))+o(d(x,x_0)).\]

Of course, this is just a definition, and there is no reason a function has to be differentiable anywhere.  Thus, we focus on {\it Lipschitz differentiability spaces} (or LDS for short), the class of metric measure spaces that admit a fixed Lipschitz coordinate functions $\phi \colon X \to \R^n$ with which one can differentiate any real valued Lipschitz function almost everywhere (a precise definition that permits a countable decomposition of the space is given in Definition \ref{d:LDS}).  In other words, Rademacher's theorem holds in LDS in a very concrete sense.

One may worry that Lipschitz differentiability spaces are a trivial class of metric measure spaces (for example, only differentiable manifolds).   Cheeger \cite{cheeger} showed that any doubling metric measure space satisfying a Poincar\'e inequality (a so called {\it PI space}) is a LDS.  Moreover, there have been many exotic constructions of PI spaces that exhibit very non-manifold behavior \cites{bourdon-pajot,laakso,cheeger-kleiner}, and so we see that LDS contain a rich class of nontrivial metric measure spaces.

However, the definition of a LDS does not require an underlying assumption of a Poincar\'e inequality and in some senses appears unecessarily strong.  When considering such spaces one is quickly lead to ask if it is a necessary condition.  Indeed, questions of this nature have been asked ever since LDS first appeared in their fullest generality in an article by the first author of this paper \cite{bate}.  For example, one variant of this question appears in \cite{cheeger-kleiner-schioppa} (see more below).  One of the main results of this paper is to prove that a Poincar\'e type inequality is in fact necessary.

In classical setting, the Poincar\'e inequality has been fundamental to many results in the theory of partial differential equations and the calculus of variations via Sobolev spaces.  For example, the Poincar\'e inequality shows that the Sobolev norm is equivalent to the norm of the weak derivative in $H^1$.  Since the latter is given by an inner product, this leads to weak solutions of the Dirichlet problem (among others) via the Riesz representation theorem.  The Poincar\'e inequality can also be used to deduce various embedding and regularity results.  Further, the best constant in the Poincar\'e inequality corresponds precisely to the minimum eigenvalue of the Laplacian.  The Poincar\'e inequality was also used in De Giorgi's proof of Hilbert's nineteenth problem.

PI spaces have also been extensively used in developments of analysis in non-Euclidean settings.  For example in the Riemannian setting, Saloff-Coste also showed that the Harnack inequality for a Riemannian manifold is equivalent to supporting a doubling measure and a 2-Poincar\'e inequality \cites{saloff-coste-1,saloff-coste-2}.  More generally, Heinonen-Koskela \cite{heinonen-koskela} introduced the notion of a Poincar\'e inequality in metric measure spaces and used it to develop the theory of quasiconformal maps in that setting.  More recently, the Poincar\'e inequality has been used to establish various equivalent generalizations of Sobolev spaces to metric measure spaces, see \cite{shanmu}.  We refer the reader to \cite{gigli} for a thorough treatise of first order calculus and of the Poincar\'e inequality on metric measure spaces.

Doubling is usually enough to import many crucial zero-order calculus statements like Vitali covering theorem, Lebesgue differentiation theorem, various maximal inequalities and so on.  In some sense, a space satisfying the Poincar\'e inequality means that rectifiable curves control the geometry of the space (see Definition \ref{d:PI}).  One can perform calculus along these curves and then try to relate them to analytic statements on the entire space.

In the setting of LDS, specifically in \cite{cheeger} and in the later refinement by Keith \cite{keith-lip-lip}, the connection to rectifiable curves was not used.  Instead, the Poincar\'e inequality was used in a way more reminiscent of the classical version: to obtain local control of a Lipschitz function entirely based upon its infinitesimal behaviour.  Moreover, this control is independent of the Lipschitz function.

The fact that this control is independent of the Lipschitz function is crucial in Cheeger and Kleiner's important extension of \cite{cheeger}.  In  \cite{cheeger-kleiner-rnp} they show that the same coordinate functions on a PI space can be used to differentiate Lipschitz functions taking value in any Banach space with the Radon-Nikodym property (RNP).  Recall that a Banach space $V$ has RNP if every Lipschitz function $f \colon \R \to V$ is differentiable (in the usual sense) almost everywhere.  This makes RNP Banach spaces a natural set of target spaces in the theory of Lipschitz differentiation.  We call the spaces satisfying Cheeger differentiability for RNP valued Lipschitz functions RNP Lipschitz differentiability spaces (or simply a RNP-LDS).  One trivially has that RNP-LDS are LDS as $\R$ has RNP.


Note that we must be slightly careful with how we ask about the necessity of a Poincar\'e inequality because, strictly speaking, there are trivial examples of LDS that do not satisfy a Poincar\'e inequality.  Since being a LDS is a property that passes to arbitrary positive measure subsets, a LDS can be totally disconnected like a fat Cantor set.  On the other hand, a Poincar\'e inequality is much more quantitative in nature.  For example, a result of Semmes (given in Cheeger's paper \cite{cheeger}) says that PI spaces must be path connected.

In this paper, we introduce the notion of an asymptotic nonhomogeneous Poincar\'e inequality or asymptotic NPI (see Definition \ref{d:ANPI}), a relaxed local version of the Poincar\'e inequality, and show that RNP-LDS must satisfy an asymptotic NPI.  This version of the Poincar\'e inequality makes sense even for disconnected spaces and so is a natural candidate to consider.  On the other hand, we do not throw away too much in this relaxation as we show that an asymptotic NPI, together with pointwise doubling, is enough to imply RNP Cheeger differentiability.  This gives us the following characterization of RNP-LDS:
{
\renewcommand{\thetheorem}{\ref{t:poincarechar}}
\begin{theorem}
  A metric measure space is a RNP-LDS if and only if it satisfies an asymptotic nonhomogeneous Poincar\'e inequality and all porous subsets have measure zero.
\end{theorem}
\addtocounter{theorem}{-1}
}
Note that the fact that all porous subsets have measure zero (see Definition \ref{def:porous} for the definition of a porous set) is a slightly stronger condition than a local version of the doubling condition on a measure.
Thus, first order calculus for RNP-valued Lipschitz functions requires---and is in fact equivalent to---relaxed local versions of doubling and the Poincar\'e inequality.  This is the first result showing the necessity of the Poincar\'e inequality (or a variant of it) for the existence of first order calculus in the general metric setting.

As demonstrated by Cheeger \cite{cheeger}, it is natural to study Gromov-Hausdorff tangents of spaces that satisfy a Poincar\'e inequality.  In the case of a RNP-LDS, we show that passing to a tangent eliminates the possibly disconnected nature of the space and that the asymptotic NPI improves to a nonasymptotic version.  This improvement is enough to allow us to deduce that the tangent is actually quasiconvex.
{
\renewcommand{\thetheorem}{\ref{t:tangentsareAPI}}
\begin{theorem}
  Let $(X,d,\mu)$ be a RNP-LDS.  Then for $\mu$-a.e. $x \in X$, any Gromov-Hausdorff tangent of $X$ at $x$ is quasiconvex and satisfies a nonhomogeneous Poincar\'e inequality (and so is also a RNP-LDS).
\end{theorem}
\addtocounter{theorem}{-1}
}

Given Theorem \ref{t:poincarechar}, it was natural to ask how close a space having the nonhomogeneous Poincar\'e inequality was to achieving a true Poincar\'e inequality.  The nonhomogeneous Poincar\'e inequality resembles the Orlicz-Poincar\'e inequality, a slightly relaxed form of the Poincar\'e inequality, that has been studied in previous work \cites{heikkinen1,tuominen1,dejarnett}.  The NPI is weaker than the Orclicz-Poincar\'e inequality as the modulus is applied only to the average of the upper gradient rather than taking an Orlicz norm.  However, we did not have an example of a space that showed that the two notions were truly different.  A natural question then was to ask whether the NPI could be improved to Orlicz or even true Poincar\'e inequality.  Recently, Eriksson-Bique developed our methods further and resolved this problem by showing that spaces with an asymptotic nonhomogeneous Poincar\'e inequality are simply countable unions of subsets of true PI spaces \cite{eriksson-bique}.  This essentially reduces the study of RNP-LDS to the study of PI spaces.

We now turn our attention to another description of RNP-LDS.  Note that the condition of a space being a LDS is described in terms of functions on that space.  To verify the condition, one would need to check almost everywhere differentiability for every Lipschitz functions defined on that space.  It would be useful to have a condition guaranteeing Cheeger differentiability that is purely geometric.  This would allow us to check only a property intrinsic to the metric measure space itself instead of the Lipschitz functions on that space.

In \cite{bate}, a necessary geometric condition is given for a space to be a LDS (and thus a RNP-LDS) by showing that they admit a collection of \emph{Alberti representations} of the underlying measure.  An Alberti representation (see Section \ref{s:albertireps} and Definition \ref{d:AR}) is a Fubini-like disintegration of the measure into 1-rectifiable measures in the metric space.  They first appeared in the work of Alberti-Cs\"ornyei-Preiss \cite{acp} that characterized those measures on Euclidean space that satisfy the conclusion of Rademacher's theorem.  Analogous to how one uses Fubini's theorem on $\mathbb R^n$ to form a partial derivative of any Lipschitz function at almost every point, the existence of an Alberti representation provides a rich family of fragmented curves to form partial derivatives in the metric setting.  We also refer the reader to the work of Alberti and Marchese \cite{alberti-marchese} for further uses of Alberti representations (there called \emph{1-rectifiable measures}), and other questions related to Rademacher's theorem, in Euclidean spaces.

Simply having several Alberti representations (and hence a \emph{gradient} of several partial derivatives at almost every point) is not sufficient to guarantee Lipschitz differentiability, however; one needs some additional condition to describe when the gradient partial derivatives is actually a total derivative.  In \cite{bate}, a characterization of Lipshitz differentiability spaces was given by the requirement that the Alberti representations be \emph{universal} (see Definition \ref{d:universal}).  Roughly speaking, a collection of Alberti representations are universal if, for every real valued Lipschitz function, the norm of the gradient of partial derivatives is comparable to the pointwise Lipschitz constant at almost every point.

Note, however, that universality of the Alberti representations is functional in nature and having a geometric description of this phenomenon would be much more satisfying.  Indeed, such a geometric description is one of the motivations of \cite{acp} and \cite{alberti-marchese} in Euclidean space.  In this paper, we introduce the purely geometric condition of {\it connecting points} in a metric space by Alberti representations and show that it is equivalent to the space satisfying RNP Cheeger differentiability.
{
\renewcommand{\thetheorem}{\ref{t:differentiability}}
\begin{theorem}
  A metric measure space $(X,d,\mu)$ is a RNP Lipschitz differentiability space if and only if it admits a countable Borel decomposition $X = \bigcup_i U_i$ such that each $\mu|_{U_i}$ has a collection of $n_i$ Alberti representations that connect points.
\end{theorem}
\addtocounter{theorem}{-1}
}

Intuitively, connecting points means that almost any point can be connected to any other sufficiently close point via a concatenation of fragmented curves that form the Alberti representations.  Here, we insist that the total length of the concatenated fragments is not too long and that the gaps in the concatenated fragments can be made arbitrarily small relative to the length if we choose the target point sufficiently close.  This is similar in spirit to how one can travel in a positive measure $S\subset \mathbb R^n$ from almost every point to any other sufficiently close point by concatenations of axis parallel line segments that travel \emph{mostly through $S$}.  A less trivial example of the same phenomenon is how one can travel from one point in the Heisenberg group to any other point via concatenating line segments parallel to the $x$ and $y$ directions.  Theorem \ref{t:differentiability} answers the question on \emph{accessibility} in RNP-LDS asked in \cite{cheeger-kleiner-schioppa} for the main space (rather than a tangent as discussed there).

One possible proof of Rademacher's theorem in $\R^n$ is to use this connectivity property and to apply the fundamental theorem of calculus along the concatenation of line segments.  With some standard measure theoretic ideas, it is then possible to show that the partial derivatives do indeed form a total derivative almost everywhere.  This idea is exactly the motivation for our definition of Alberti representations connecting points, and the backwards direction of Theorem \ref{t:differentiability} uses essentially the same strategy.  Thus, one can view Theorem \ref{t:differentiability} as an explanation of how the RNP Cheeger differentiation phenomenon works.  Concretely, a space is a RNP-LDS precisely because the fragmented curves from its Alberti representations can be concatenated to connect points in a way so that the fundamental theorem of calculus, and hence this proof of Rademacher's theorem, can be applied.

This also clarifies the role of RNP for the target space in the characterization of Theorem \ref{t:differentiability}.  The proof of Rademacher's theorem relies on combining partial derivatives of $f$ together to become a full derivative.  In our setting, partial derivatives will be interpreted as the derivative of $f$ along curve fragments from the Alberti representations.  Thus, the proof will work as long as we can take derivatives of $f$ along curves, and RNP Banach spaces are precisely the class of Banach spaces for which this can happen.  

The relationship between a rich structure of rectifiable curves and Poincar\'e inequalities has been extensively studied in the PI setting.  It is known that, under some mild additional assumptions, a metric measure space is a PI space if and only if every two continua can be connected by a collection of curves with \emph{large modulus}
(c.f. \cites{heinonen,keith-mod}).  Further, a common technique of constructing a PI space uses a method of Semmes \cite{semmes} that involves joining any two points in the space by a \emph{thick pencil of curves}.  One may consider the relationship between an asymptotic NPI and connecting points by Alberti representations as a qualitative, disconnected generalization of these statements.

Cheeger's formulation of Lipschitz differentiability can be interpreted as the existence of a finite dimensional measurable cotangent bundle $T^*X$.  In this formulation, Lipschitz functions give sections of $T^*X$ via their differentials.  One can also view Alberti representations as a generalization of a measurable vector field given by the ``derivatives'' of the curve fragments in the representations.  Indeed, this was first demonstrated by Schioppa \cite{schioppa-derivations}, where it is shown that, by averaging the derivative along curve fragments that pass through a given point, there is a correspondence between Alberti representations and \emph{Weaver derivations} \cites{weaver1, weaver2}.  Further, the notion of a measurable tangent bundle $TX$ was developed by Cheeger, Kleiner and Schioppa \cite{cheeger-kleiner-schioppa} and was used to deduce a theory of \emph{metric differentiation}, akin to Kirchheim and Ambrosio-Kirchheim \cites{kirchheim,ambrosio-kirchheim}, of arbitrary Lipschitz functions defined on a LDS.

This interpretation must be treated somewhat loosely, because an Alberti representation need not assign unique vectors to points since there a point can lie in multiple curves of a representation.  However, if we continue with this analogy, then connecting points can be viewed as a collection of Alberti representations generating the entire Lie algebra in the sense that one can travel anywhere locally by concatenating the curves coming from the Alberti representations.  In this sense, connecting points can also be viewed as a measurable analogue to H\"ormander's condition.  Recall smooth vector fields $X_1,...,X_n$ satisfy H\"ormander's condition if they generate the entire Lie algebra of vector fields.  Of course in a Lie algebra, the generation is done under Lie brackets, for which there is no real analogue of in our setting.

Since the only hypothesis on a RNP-LDS is that Lipschitz functions are differentiable almost everywhere, the proof of our results must rely on constructing a Lipschitz function that is not differentiable on some set and so conclude that such a set has measure zero.  From this point of view, the start of our approach is similar to that in \cite{bate}.  However, there is one crucial difference which ultimately leads to our requirement of an RNP valued Lipschitz function (compared to real valued).  In \cite{bate}, one is given some auxilliary Lipschitz functions $f_i\colon X\to \R$ that exhibit some bad behaviour on a set $S$ and use them to construct a Lipschitz function that is not differentiable on $S$.  In our case, however, due to the pointwise behaviour of connecting points by curves, the auxilliary functions only exhibit bad behaviour on a small neighbourhood of a single point.  Although every point sees the bad behaviour of some function, these functions cannot be combined to produce a \emph{single} real valued Lipschitz function that encompasses every point.  Therefore, we must resort to constructing Lipschitz functions that take values in larger spaces.

Finally we briefly discuss whether these characterizations hold for LDS, that is spaces for which Cheeger differentiability is only {\it a priori} known for real valued Lipschitz functions.  Equivalently, whether RNP-LDS is a strict subclass of LDS or not.  For a long time, the only known examples of LDS come from Cheeger's result and consist of PI spaces or countable unions of subsets thereof.  By Cheeger and Kleiner's result, this automatically upgrades them to RNP-LDS.  An initial motivation for studying spaces with Alberti representations that connect points was to attempt to construct a LDS that does not possess any version of a Poincar\'e inequality.  However, given the results presented here, this approach cannot work since such a space must be a RNP-LDS and hence satisfy some type of Poincar\'e inequality.

In this paper we do not try to address this issue and we will exclusively deal with RNP-LDS.  However, after the first preprint of this paper appeared, Schioppa \cite{schioppa-unrect} gave a construction of a LDS that is not a RNP-LDS, and hence none of our characterizations can possibly hold for this space.
Thus, this completes one relationship between Poincar\'e inequalities and differentiability in metric measure spaces: RNP-LDS are equivalent to spaces with asymptotic NPI, and by the results of Eriksson-Bique this is equivalent to being a countable union of subsets of PI spaces.  Further, Schioppa's example shows that this is strictly stronger than simply being a LDS.

We now briefly describe the outline of this paper.  In section \ref{s:preliminaries} we first give the required background on Poincar\'e inequalities, Lipschitz differentiability spaces and Alberti representations and the relationships between them.  We also introduce the necessary standard measure theoretic ideas and constructions.

Our main arguments begin in section \ref{s:findingcurves} by showing the existence of many rectifiable curve fragments that connect points in a RNP-LDS.  Extending upon this, section \ref{s:poincare} uses these families of curves to construct an asymptotic NPI in such a space.

This argument requires a construction of a non-differentiable RNP valued Lipschitz function, the details of which we defer until section \ref{s:nondiff}.  We reiterate that, in contrast to the construction in \cite{bate}, this construction uses, as its basic building block, functions that exhibit bad behaviour only at a single point (rather than a set of full measure).  

We conclude this study of NPI spaces in section \ref{s:tangents} by studying the limiting behaviour of this inequality into a Gromov-Hausdorff tangent at almost every point.

We then proceed to our second characterisation of an RNP-LDS, one which is based upon applying the fundamental theorem of calculus along curves that define a partial derivative of any given Lipschitz functions.  For this we first introduce the notion of a \emph{restricted} Alberti representation in section \ref{s:restrictedcurves}.  This concept allows us to just consider some of the curves that form an Alberti representation, and hence consider all the possible paths that may be constructed by concatenating such curves and, in particular consider those points that may be connected using such paths.

Finally, in section \ref{s:diff}, we use restricted Alberti representations to characterise exactly those metric measure spaces that are RNP-LDS.

\subsection{Acknowledgements}
The idea of using Alberti representations to join points to form the derivative in a Lipschitz differentiability space arose in conversations between the first author and David Preiss and later with the authors and Marianna Cs\"ornyei.  The idea of finding or disproving the existence of a Poincar\'e inequality in Lipschitz differentiability spaces was discussed in various conversations with David Preiss, Bruce Kleiner and Jeff Cheeger.  We are very grateful for the insight gained from these conversations.  We are also extremely grateful to the anonymous referee for a detailed report that highlighted several ways we could improve our exposition, and spotted an error in the definition of asymptotic NPI spaces, which has been corrected.  S. Li was supported by a postdoctoral research fellowship NSF DMS-1303910.

\section{Preliminaries}\label{s:preliminaries}
Throughout this paper, a metric measure space $(X,d,\mu)$ will refer to a complete and separable metric $(X,d)$ equipped with a finite Borel regular measure $\mu$.  Almost all of our main results immediately apply to a general metric space with a Radon measure by exhausting such spaces by a countable union of metric measure spaces of the above form.  The one exception is Theorem \ref{t:differentiability} which we cover in Remark \ref{r:connecting-char-Radon}.

There are several natural notions that arise when discussing Lipschitz functions.  Let $f : (X,d_X) \to (Y,d_Y)$ be a function and $x \in X$.  We then define
\begin{align*}
  \Lip(f,x_0) &:= \limsup_{r \to 0} \sup \left\{ \frac{d_Y(f(x),f(x_0))}{r} : 0 < d(x,x_0) < r\right\} \\
  &~= \limsup_{x \to x_0} \frac{d_Y(f(x),f(x_0))}{d_X(x,x_0)}, \\
  \lip(f,x_0) &:= \liminf_{r \to 0} \sup \left\{ \frac{d_Y(f(x),f(x_0))}{r} : 0 < d(x,x_0) < r\right\}
\end{align*}
Note that $\Lip$ satisfies a chain rule inequality for Lipschitz functions, that is
\begin{align*}
  \Lip(f \circ g,x) \leq \Lip(f,g(x)) \Lip(g,x).
\end{align*}

A metric space $(X,d)$ is \emph{metrically doubling} if there exist some $M \geq 1$ so that for all $x \in X$, $r > 0$, there exist some $\{y_1,...,y_n\} \subseteq X$ where $n \leq M$ so that
\begin{align*}
  B(x,r) \subseteq \bigcup_{i=1}^n B(y_i,r/2).
\end{align*}
Further, a metric measure space $(X,d,\mu)$ is \emph{doubling} if there exists some $C > 0$ so that for every $x \in X$ and $r > 0$, we have that
\begin{align*}
  \mu(B(x,r)) \leq C \mu(B(x,r/2)).
\end{align*}
By a standard volume packing argument, it easily follows that doubling metric measure spaces are also metrically doubling.

There is also an asymptotic version of a doubling measure.  A metric measure space is pointwise doubling at $x \in X$ if
\begin{align*}
  \limsup_{r \downarrow 0} \frac{\mu(B(x,2r))}{\mu(B(x,r))} < \infty.
\end{align*}
A metric measure space is \emph{pointwise doubling} if it is pointwise doubling almost everywhere.  By the standard argument for doubling measures, any pointwise doubling metric measure space is a Vitali space and in particular satisfies the Lebesgue density theorem (see Theorem 3.4.3 of \cite{hkst} for the standard argument).  It follows that any of its subset also is pointwise doubling.  We say that a subset $A \subset X$ is $(C,R)$-uniformly pointwise doubling if
\begin{align}
  \mu(B(x,r)) \leq C \mu(B(x,r/2)), \qquad \forall r < R, x \in A. \label{e:unif-pointwise-doubling}
\end{align}
Most doubling type arguments only require uniformly pointwise doubling.  Specifically, the conclusion of the following lemma.
\begin{lemma} \label{l:doubling-balls}
  Let $C > 1$, $R > 0$, and $A \subset X$ be $(C,R)$-uniformly pointwise doubling.  Then for every $x \in A$, $y \in X$ and $r'>0$ such that $x \in B(y,r') \subset B(x,R)$ and for every $r < R$, we have
  \begin{align*}
    \frac{\mu(B(x,r))}{\mu(B(y,r'))} \geq 4^{-s} \left( \frac{r}{r'} \right)^s.
  \end{align*}
  Here, $s = \log_2 C > 0$.
\end{lemma}

The following lemma shows that we can decompose pointwise doubling metric measure spaces to a countable collection of metric doubling subsets.
\begin{lemma}[\cite{bate}*{Lemma 8.3}] \label{l:pointwise-doubling-decomp}
  If $(X,d,\mu)$ is pointwise doubling, then there exist a countable number of Borel subsets $A_i \subseteq X$ such that $\mu\left(X \backslash \bigcup_i A_i \right) = 0$ and $(A_i,d)$ is metrically doubling.
\end{lemma}
One can also easily decompose a pointwise doubling metric measure space into uniformly pointwise doubling pieces.
\begin{lemma} \label{l:unif-pointwise-doubling-decomp}
  If $(X,d,\mu)$ is pointwise doubling, then there exist a countable number of Borel subsets $A_i \subseteq X$ such that $\mu\left(X \backslash \bigcup_i A_i \right) = 0$ and constants $C_i > 1$, $R_i > 0$ so that $(X,d,\mu)$ is $(C_i,R_i)$-uniformly pointwise doubling at each $x \in A_i$.
\end{lemma}
\begin{remark} \label{r:unif-pointwise-doubling}
  In Lemma \ref{l:unif-pointwise-doubling-decomp}, we are not saying that each $(A_i,d,\mu)$ are uniformly pointwise doubling, but that \eqref{e:unif-pointwise-doubling} is satisfied for balls in $X$ centered in $A_i$.
\end{remark}

When dealing with (pointwise) doubling measures, it is natural to consider porous sets.
\begin{definition}\label{def:porous}
A set $S \subseteq (X,d,\mu)$ is {\it $\eta$-porous} at $x \in S$ if there exist a sequence $(y_i) \subset X$ such that $y_i \to x$ and
\begin{align*}
  S \cap B(y_i,\eta d(x,y_i)) = \emptyset.
\end{align*}
It is said to be porous at $x$ if it is $\eta$-porous at $x$ for some $\eta > 0$.  Finally, a set $S \subseteq (X,d,\mu)$ is \emph{porous} if it is porous at each of its points.
\end{definition}
Indeed, by \cite{porosityandmeasures}, any metric measure spaces in which porous sets are null are automatically pointwise doubling.

A very useful fact is that we may pass to a subset of a metric space and the pointwise behaviour of a Lipschitz function will only differ at a porosity point of the subset.

\begin{lemma}\label{l:porosity-subset}
  Let $(X,d)$ and $(Y,d')$ be metric spaces, $f\colon X\to Y$ Lipschitz and $x\in S\subset X$.  If
  \begin{equation}\label{e:porosity-subset-1}\limsup_{X\ni y\to x}\frac{d'(f(y),f(x))}{d(y,x)}> \limsup_{S\ni y\to x}\frac{d'(f(y),f(x))}{d(y,x)}\end{equation}
  then $x$ is a porosity point of $S$.

  In particular, if $X$ is equipped with a measure for which all porous sets have measure zero, then
  \begin{equation}\label{e:porosity-subset-2}\limsup_{X\ni y\to x}\frac{d'(f(y),f(x))}{d(x,y)}= \limsup_{S\ni y\to x}\frac{d'(f(y),f(x))}{d(x,y)}\end{equation}
  for $\mu$-a.e. $x\in S$.
\end{lemma}

\begin{proof}
  If \eqref{e:porosity-subset-1} is true then there exists $y_n\to x$ and an $\eta>0$ such that
  \[\lim_{n\to \infty}\frac{d'(f(y_n),f(x))}{d(y_n,x)}> \eta+\limsup_{S\ni y\to x}\frac{d'(f(y),f(x))}{d(x,y)}.\]
  In particular, the triangle inequality shows that, for sufficiently large $n$, $B(y_n,\eta d(y_n,x)/2\|f\|_{lip})$ is disjoint from $S$.  Thus $S$ is porous at $x$.

  Finally, if $X$ has such a measure and \eqref{e:porosity-subset-2} does not hold at some $x\in S$ then \eqref{e:porosity-subset-1} holds for $x$ (since a supremum over $S$ is certainly less than or equal to a supremum over $X$).  Therefore $S$ is porous at $x$ and so the set of all such $x$ is a porous subset of $S$ and so has measure zero.
\end{proof}

\subsection{Lipschitz differentiability spaces}\label{s:preliminarylds}
We now discuss a notion of differentiation on metric measure spaces.  These definitions were first present in the work of Cheeger \cite{cheeger} and were later refined by Keith \cite{keith-lip-lip}.

We will say that a Borel $U \subseteq X$ and Lipchitz $\phi \colon U \to \R^n$ form a \emph{chart of dimension $n$} $(U,\phi)$.  Now let $V$ be a Banach space and $f \colon X \to V$.  We say that $f$ is differentiable at $x_0\in U$ with respect to $(U,\phi)$ if there exists a {\it unique} linear transform $Df(x_0) \in L(\R^n,V)$ such that
\begin{align*}
  \lim_{X \ni x \to x_0} \frac{\|f(x) -f(x_0) - Df(x_0)(\phi(x) - \phi(x_0))\|}{d(x,x_0)} = 0.
\end{align*}
We then say that $Df(x_0)$ is the (Cheeger) derivative of $f$ at $x_0$.  We now define the classes of metric spaces for which there is a rich differentiability structure.

\begin{definition} \label{d:LDS}
  A metric measure space $(X,d,\mu)$ is said to be a \emph{Lipschitz differentiability space (LDS)} if there exist a countable collection of charts $(U_i,\phi_i)$ so that $X = \bigcup_i U_i$ and every Lipschitz map $f \colon X \to \R$ is differentiable at almost every point with respect to every chart.  If all charts $\phi_i$ take value in $\R^n$, then we say that $X$ is an $n$-dimensional Lipschitz differentiability space.
\end{definition}
We similarly define a subclass of Lipschitz differentiability spaces as follows.
\begin{definition}
  A metric measure space $(X,d,\mu)$ is said to be a \emph{Radon-Nikodym Lipschitz differentiability space (RNP-LDS)} if there exist a countable collection of charts $(U_i,\phi_i)$ with $X = \bigcup_i U_i$ such that, for every Banach space $V$ with the Radon-Nikodym property, every Lipschitz map $f : X \to V$ is differentiable at almost every point with respect to every chart.
\end{definition}

Recall that a Banach space $V$ has the Radon-Nikodym property (RNP) if every Lipschitz $\gamma \colon \R \to V$ is differentiable almost everywhere in the usual sense:
\[\lim_{h\to 0} \frac{\gamma(t+h)-\gamma(t)}{h}\]
exists for $\mathcal L^1$-a.e. $t\in \R$.  Note that, by Rademacher's theorem, $\R$ has the RNP and so we necessarily have that every RNP-LDS is a LDS.  Also, by definition, RNP Banach spaces are the largest subclass of Banach space for which one can hope to get nontrivial Lipschitz differentiability behavior.

In \cite{cheeger} Cheeger introduced the notion of Lipschitz differentiability with respect to charts in metric spaces and showed that PI spaces (see below) are Lipschitz differentiability spaces (though, without this name).  Later, it was shown by Cheeger and Kleiner \cite{cheeger-kleiner-rnp} that PI spaces are in fact RNP-Lipschitz differentiability spaces.

We now recall some facts on the nature of LDS (and hence of RNP-LDS) that we will require.
\begin{theorem}[\cite{bate-speight}*{Theorem 2.4 and Corollary 2.6}] \label{t:porous-null}
  Any porous set $S$ in a Lipschitz differentiability space is null.  In particular a Lipschitz differentiability space is pointwise doubling.
\end{theorem}

A consequence of this is that subsets of a LDS are also a LDS, see \cite{bate-speight}*{Corollary 2.7}.  The same is true for a RNP-LDS.

\begin{lemma} \label{l:RNPLDS-subset}
  Let $(X,d,\mu)$ be a RNP Lipschitz differentiability space with charts $(U_i,\phi_i)$.  Then for any measurable $S \subseteq X$, $(S,d,\mu)$ is also a RNP Lipschitz differentiability space with charts $(U_i \cap S,\phi_i)$.
\end{lemma}

\begin{proof}
Let $V$ be a RNP Banach space and $f : S \to V$ be Lipschitz.  By the case covered in \cite{bate-speight}, we know that $f$ is differentiable almost everywhere whenever $V$ has finite dimension.  Thus we may suppose that $V$ is infinite dimensional.

By Lemma \ref{l:pointwise-doubling-decomp} and Theorem \ref{t:porous-null}, we can get $\{A_j\}$ a countable set of metrically doubling subsets of $X$ such that $\mu\left(X \backslash \bigcup_j A_j\right) = 0$.  We first prove the result for $S'$ contained within a single $A_j$ and a single chart $U_i$.

As $S'$ lies in the metrically doubling $A_j$, it is itself metrically doubling.  Thus, by applying Theorem 1.6 of \cite{lee-naor} to $(S',X)$, we get that there exists a Lipschitz function $F : X \to V$ such that $F|_{S'} = f|_{S'}$.

By assumption, for almost all $x_0 \in S'$, we have some $DF(x_0) \in L(\R^n,V)$ so that
  \begin{align*}
    \lim_{X \ni x \to x_0} \frac{\|F(x) - f(x_0) - DF(x_0)(\phi(x)-\phi(x_0))\|}{d(x,x_0)} = 0.
  \end{align*}
In particular,
  \begin{align*}
    \lim_{S' \ni x \to x_0} \frac{\|f(x) - f(x_0) - DF(x_0)(\phi(x)-\phi(x_0))\|}{d(x,x_0)} = 0.
  \end{align*}
That is, $DF(x_0)$ is \emph{a} derivative of $F$ at $x_0$, we must show uniqueness.

As shown in \cite{bate-speight} (the same argument works for vector valued functions), uniqueness of the derivative at a point $x_0$ within a subset $S'$ is equivalent to the existence of a $\lambda>0$ such that
\[\limsup_{S'\ni x\to x_0} \frac{|(\phi(x)-\phi(x_0))\cdot v|}{d(x,x_0)} \geq \lambda \|v\|\]
for every $v\in\mathbb S^{n-1}$.  Since $(S',d,\mu)$ is a LDS by \cite{bate-speight}, this condition must be true for $\mu$-a.e. $x_0\in S'$ and so $(S',d,\mu)$ is also a RNP-LDS.

To prove the general case, for any measurable $S\subset X$ there exists a countable Borel decomposition of $S$ into disjoint sets $S_i$ and a null set such that each $S_i$ has the form of $S'$ above.  Given any $f\colon S\to V$, since each $S_i$ is a RNP-LDS, for every $i$ and almost every $x_0\in S_i$, we have a derivative $Df(x_0)$ of $f$ within $S_i$.  If $Df(x_0)$ is not a derivative of $f$ within $X$ then
\begin{multline*}\limsup_{X\ni x\to x_0}\frac{\|f(x_n)-f(x_0)-Df(x_0)(\phi(x)-\phi(x_0))\|}{d(x,x_0)}>0\\ = \limsup_{S\ni x\to x_0}\frac{\|f(x_n)-f(x_0)-Df(x_0)(\phi(x)-\phi(x_0))\|}{d(x,x_0)}.\end{multline*}
Thus, by Lemma \ref{l:porosity-subset}, the set of all such $x_0$ has measure zero.
\end{proof}

\subsection{Alberti representations}\label{s:albertireps}
Recall from \cite{bate} the set of \emph{curve fragments}
\[\Gamma = \{\gamma \colon K \subset \R \to X : K \text{ is compact, } \gamma \text{ is 1-Lipschitz}\}.\]
For $\gamma\in \Gamma$, the \emph{metric derivative} of $\gamma$ given by
\[\md_\gamma(t) = \lim_{h\to 0}\frac{d(\gamma(t+h),\gamma(t))}{|h|}\]
exists almost everywhere in the domain of $\gamma$ (see the following remark).  Note that, if it exists, $\md_\gamma(t)=\Lip(\gamma,t)$.  We identify each $\gamma \in \Gamma$ with a subset of $X\times \R$ via $\gamma \mapsto \{(\gamma(t),t): t\in \dom \gamma\}$ (the \emph{graph} of $\gamma$) and consider $\Gamma$ with the induced Hausdorff metric.

\begin{remark}\label{rmk:ambrosio}
The metric derivative of a Lipschitz function $\gamma\colon K\subset \R \to X$ was introduced in \cite{ambrosio} and represents the infinitesimal speed of $\gamma$.  If $K$ is an interval, it is shown that it exists almost everywhere and its integral is the arc length of $\gamma$.  However, if $K$ is a compact set we may isometrically embed $X$ into $\ell^\infty$ and extend $\gamma$ linearly on the complement of $K$ so that its domain is an interval.  Then at any density point of $K$ the limit in $K$ agrees with the limit of the extended function and so also exists almost everywhere.
\end{remark}

In \cite{bate}, the curve fragments were biLipschitz, but this will not be a matter as we will restrict to biLipschitz fragments in the following definition.  \begin{definition} \label{d:AR}
  Let $\bP$ be a probability measure on $\Gamma$ for which the set of non-biLipschitz fragments have measure zero.  For each biLipschitz $\gamma\in \Gamma$ let $\mu_\gamma$ be a measure such that $\mu_\gamma \ll \Hd^1 \llcorner \im \gamma$.  Then the pair $(\bP,\{\mu_\gamma\})$ is an \emph{Alberti representation} of $\mu$ if for every Borel set $B \subset X$, we have that
  \begin{align*}
    \mu(B) = \int_\Gamma \mu_\gamma(B) ~d\bP.
  \end{align*}
\end{definition}

If $V$ has the RNP and $f\colon X \to V$ is Lipschitz, then the existence of an Alberti representation guarantees, for almost every $x_0\in X$, a biLipschitz $\gamma\in \Gamma$ and a $t_0$ a density point of $\dom \gamma$ such that $\gamma(t_0)=x_0$ and such that $(f\circ\gamma)'(t_0)$ exists (see \cite{bate}*{Proposition 2.9}).  The existence of such a \emph{partial derivative} is precisely the motivation for relating Alberti representations and Lipschitz differentiability spaces.

Next we introduce some properties of Alberti representations.  Given a Lipschitz function $\phi : X \to \R^n$, we say that an Alberti representation of $\mu$ is in the \emph{$\phi$-direction} of a cone $C \subset \R^n$ if
\begin{align*}
  (\phi \circ \gamma)'(t) \in C \backslash \{0\}, \qquad \bP\mbox{--}a.e. ~\gamma \in \Gamma, \mu_\gamma\mbox{--}a.e. ~t \in \dom \gamma.
\end{align*}
A set of cones $C_1,...,C_k$ in $\R^n$ is said to be linearly independent if for any collection of nonzero vectors $v_i \in C_i$, we have that $\{v_1,...,v_k\}$ is linearly independent.  A collection of Alberti representations $\cA_1,...,\cA_n$ is then said to be \emph{$\phi$-independent} if there exist $n$ linearly independent cones $C_1,...,C_n$ in $\R^n$ such that each $\cA_i$ is in $\phi$-direction $C_i$.

Note that the existence of $n$ independent Alberti representations give rise to $n$ partial derivatives of any RNP valued Lipschiz function $f\colon X \to V$ at almost every point.  Precisely, let $(U,\phi)$ be an $n$-dimensional chart and $x \in U$.  Suppose there exist $\gamma_1,...,\gamma_n \in \Gamma$ so that $\gamma_i^{-1}(x) = 0$ is a density point of $\dom \gamma_i$ such that $(\phi \circ \gamma_i)'(0) \in \R^n$ exist and are linearly independent and $(f \circ \gamma_i)'(0) \in V$ also exist.  Then we can form the linear map $\nabla f(x) \in L(\R^n,\R)$ so that
\begin{align*}
  \nabla f(x) (\phi \circ \gamma_i)'(0) = (f \circ \gamma_i)'(0), \qquad \forall i.
\end{align*}
We call $\nabla f(x)$ a {\it gradient} of $f$ at $x$ with respect to $\phi$.  Note that in general, a gradient may not be unique since two different choices of the $\gamma_i$ may produce different partial derivatives.  However, this is not true in a LDS since the gradient agrees with the derivative at almost every point (see Theorem \ref{t:LDS-UAR} below).

Now let $\kappa > 0$, $V$ be a RNP Banach space, and $f : X \to V$ Lipschitz.  We say that a curve fragment $\gamma\in \Gamma$ has speed $\kappa$ (or $f$-speed $\kappa$) if for almost every $t \in \dom \gamma$
\begin{align*}
  \|(f \circ \gamma)'(t)\| \geq \kappa \Lip(f,\gamma(t)) \Lip(\gamma,t).
\end{align*}
We then say that an Alberti representation $\cA$ has \emph{speed $\kappa$} (or $f$-speed $\kappa$) if $\bP$-a.e. $\gamma \in \Gamma$ has $f$-speed $\kappa$.
\begin{definition}\label{d:universal}
Let $(U,\phi)$ be an $n$-dimensional chart in $(X,d,\mu)$ and $\cA_1,...,\cA_n$ a collection of $\phi$-independent Alberti representations of $\mu \llcorner U$, each with positive $\phi$-speed.  We say that $\cA_1,...,\cA_n$ is \emph{universal} if there exists a $\kappa > 0$ so that, for any Lipschitz function $f : X \to \R$, we have a Borel decomposition $X = X_1 \cup ... \cup X_n$ such that the Alberti representation on $\mu \llcorner U_i$ as induced by $\cA_i$ has $f$-speed $\kappa$.
\end{definition}
Similarly, we say that a collection of Alberti representations are RNP-universal if there exists {\it one} constant $\kappa > 0$ so that the above universality condition holds with constant $\kappa$ uniformly for every RNP Banach space $V$ and every Lipschitz function $f \colon X \to V$.

Universal Alberti representations were introduced in \cite{bate} as a characterization of Lipschitz differentiability:
\begin{theorem}[\cite{bate}*{Theorem 7.8}] \label{t:LDS-UAR}
  A metric measure space $(X,d,\mu)$ is a Lipschitz differentiability space if and only if there exists a countable Borel decomposition $X =\bigcup_i U_i$ such that each $\mu \llcorner U_i$ has a finite universal collection of Alberti representation.

  In this case, if $(U,\phi)$ is an $n$-dimensional chart and $\cA_1,...,\cA_n$ are Alberti representations that are universal and $\phi$-independent, then any gradient $\nabla f(x_0)$ of $f$ at $x_0$ with respect to the $\cA_i$ is the derivative of $f$ at $x_0$.
\end{theorem}

We have the analogous property for RNP Lipschitz differentiability spaces.
\begin{proposition} \label{p:RNP-universal}
  If $(X,d,\mu)$ is a RNP Lipschitz differentiability space, then there exists a countable Borel decomposition $X = \bigcup_i U_i$ such that each $\mu \llcorner U_i$ has a finite RNP-universal collection of Alberti representations.
\end{proposition}

\begin{proof}
  Let $V$ be a RNP Banach space and $T \in L(\R^m,V)$.  Then by simple linear algebra, we have that
  \begin{align*}
    \Lip(T,0) = \|T\| \leq m \max_{1 \leq i \leq m} \|T(e_i)\|.
  \end{align*}
  Let $i \in \{1,...,m\}$ be the index achieving the maximum of the right hand side.  Then there exists some $\epsilon_m > 0$ depending only on $m$ so that for all $v \in \{w \in \R^n : w \cdot e_i \geq (1 - \epsilon) |w| \} =: C(e_i,\epsilon_m)$ we have
  \begin{align}
    \|T(v)\| \geq \frac{1}{2} \|T(e_i)\| \|v\| \geq \frac{1}{2m} \Lip(T,0) \|v\|. \label{e:small-cone}
  \end{align}

  Let $(U_i,\phi_i : U_i \to \R^{n_i})$ be the countable set of charts of $X$.  By Lemma 3.7 of \cite{bate}, we may suppose that for each $\phi_i$, there exists some $\lambda_i > 0$ so that $\Lip(\phi_i,x) > \lambda_i$ for almost all $x \in U_i$.  As $X$ is also a Lipschitz differentiability space, we have by Theorem 9.5 of \cite{bate} that, for each $(U_i,\phi_i)$, there exists a collection $\cA_{i,1},...,\cA_{i,n_i}$ of Alberti representations such that each $\cA_{i,j}$ is in the $\phi_i$-direction of $C(e_j,\epsilon_{n_i})$ with $\phi_i$-speed $\kappa_i > 0$.

  Let $f \colon X \to V$ be Lipschitz.  Let $N \subset U_i$ be the set of points $x \in U$ for which $Df(x)$ does not exist.  Then for $x \in U_i \backslash N$, we put $x \in V_j$ if
  \begin{align*}
    \|Df(x)e_j\| = \max_{1 \leq k \leq n_i} \|Df(x)e_k\|,
  \end{align*}
  breaking ties arbitrarily.

  Note that as
  \begin{align}
    \lim_{X \ni y \to x} \frac{\|f(y) - f(x) - Df(x)(\phi_i(y) - \phi_i(x))\|}{d(y,x)} = 0 \label{e:universal-f-diff}
  \end{align}
  whenever $f$ is differentiable at $x$, we get for almost every $x \in U_i$ that
  \begin{align}
    \Lip(f,x) \leq \Lip(Df(x),0) \Lip(\phi_i,x). \label{e:Lip-submult}
  \end{align}

  For $\bP_j$-a.e. $\gamma \in \Gamma$ and $\mu_\gamma$-a.e. $t \in \dom\gamma$ such that $\gamma(t) \in V_j$, we get that $(\phi_i \circ \gamma)'(t)$ and $(f \circ \gamma)'(t)$ exist and $(\phi_i \circ \gamma)'(t) \in C(e_j,\epsilon_{n_j})$.  Thus, we get that
  \begin{align*}
    \|(f \circ \gamma)'(t)\| &\overset{\eqref{e:universal-f-diff}}{=} \|Df(\gamma(t))(\phi_i \circ \gamma)'(t)\| \\
    &\overset{\eqref{e:small-cone}}{\geq} \frac{1}{2n_i} \Lip(Df(\gamma(t)),0) \|(\phi_i \circ \gamma)'(t)\|  \\
    &\geq \frac{\kappa_i}{2n_i} \Lip(Df(\gamma(t)),0) \Lip(\phi_i,\gamma(t)) \Lip(\gamma,t) \\
    &\overset{\eqref{e:Lip-submult}}{\geq} \frac{\kappa_i}{2n_i} \Lip(f,\gamma(t)) \Lip(\gamma,t).
  \end{align*}
\end{proof}

\begin{remark}
  Notice that, completely analogously to the real case given in \cite[Theorem 10.4]{bate}, a collection of $\kappa$ RNP-universal Alberti representations gives rise to a Lip-lip inequality for RNP valued Lipschitz functions.  Indeed, if $f\colon X \to V$ is Lipschitz, then for almost every $x$ there exists a $\gamma$ and a $t$ with $\gamma(t)=x$ such that
  \[\kappa\Lip(f,x)\leq (f\circ\gamma)'(t) \leq \lip(f,x).\]
\end{remark}

\subsection{The Poincar\'e inequality}
Given a real valued function $u$ on a metric space $(X,d)$, we say that a Borel function $\rho : X \to [0,\infty]$ is an {\it upper gradient} for $u$ if
\begin{align*}
  |u(\gamma(b)) - u(\gamma(a))| \leq \int_\gamma \rho ~ds
\end{align*}
for every rectifiable curve $\gamma : [a,b] \to X$.  Here, the line integral on the right must be independent of reparameterization.  Note that this definition is only meaningful if there are rectifiable curves in $X$ as otherwise any $\rho$ would work.  It also follows that there always exist upper gradients, even in spaces well connected by rectifiable curves, as we can take $\rho \equiv \infty$.  If $u$ is $L$-Lipschitz, then $\rho \equiv L$ is an upper gradient, but there may be better ones.

Using upper gradients, we define a Poincar\'e inequality as follows.
\begin{definition} \label{d:PI}
  Let $p \in [1,\infty)$.  A metric measure space $(X,d,\mu)$ is said to satisfy a $p$-Poincar\'e inequality in the sense of Heinonen and Koskela if there exists a $C > 0$ and a $\lambda \geq 1$ so that
  \begin{align*}
    \fint_{B(x,r)} |u - u_{B(x,r)}| ~d\mu \leq Cr \left( \fint_{B(x,\lambda r)} \rho^p ~d\mu \right)^{1/p}, \qquad \forall x \in X, r > 0,
  \end{align*}
  whenever $u : X \to \R$ is 1-Lipschitz with upper gradient $\rho : X \to [0,1]$.  Here, $u_{B(x,r)}$ is the average of $u$ on the ball $B(x,r)$.

  A doubling metric measure space supporting a $p$-Poincar\'e inequality for some $p \in [1,\infty)$ is called a {\it PI space}.
\end{definition}

Notice also that the ball on the right hand side is inflated by a factor of $\lambda$.  By H\"older's inequality, we easily see that a $p$-Poincar\'e inequality implies a $q$-Poincar\'e inequality when $q > p$.  Also note that a simple scaling argument shows that the Poincar\'e inequality is true for any Lipschitz function.

There are other versions of the Poincar\'e inequality where $u$ is allowed to be continuous or measurable.  Under the assumption that the underlying space is proper and doubling, these notions are all equivalent \cite{heinonen-koskela-note}.  We use the Lipschitz functions version, which is {\it a priori} the weakest, because the proof of Lipschitz differentiability from the Poincar\'e inequality only needs this version.

Recall that the notion of upper gradients is vacuous in a space without rectifiable curves.  It was proven by Semmes (although the proof was given in a paper of Cheeger) that PI spaces have a rich family of rectifiable curves.  Recall that a metric space $X$ is quasiconvex if there exists some $C \geq 1$ so that, for any $x$ and $y$ in $X$, there exists a 1-Lipschitz curve $\gamma$ connecting $x$ to $y$ such that
\begin{align*}
  \len(\gamma) \leq C d(x,y).
\end{align*}
Another way to interpret this condition is that the metric of quasiconvex spaces are biLipschitz to their corresponding path metric.  We have the following theorem:

\begin{theorem}[\cite{cheeger}*{Theorem 17.1}] \label{t:PI-quasiconvex}
  PI spaces are quasiconvex.
\end{theorem}

Examples of PI spaces include $\R^n$ (classical), Carnot groups \cite{jerison}, and many fractal constructions \cites{bourdon-pajot,laakso,cheeger-kleiner}.
For more information about the Poincar\'e inequality, see \cite{heinonen} and the references contained therein.


\section{RNP-differentiability implies local connectivity}\label{s:findingcurves}

We will measure the ``path connectedness" of a metric space using the following quantities.  Recall that $\gamma \in \Gamma$ is a 1-Lipschitz function from a compact subset of $\R$ into $X$.

\begin{definition}
Let $(X,d)$ be a metric space.  For $\gamma \in \Gamma$ we define the \emph{length} of $\gamma$, $\len \gamma$ to be the length of the smallest interval that contains $\dom \gamma$.  We also define the \emph{mass} of $\gamma$ to be $\mathcal L^1(\dom\gamma)$.

For a function $\rho \colon X \to [0,1]$ we define
\[\int_\gamma^* \rho = \left(\len\gamma - \ms\gamma\right)+ \int_{\dom\gamma} \rho \circ \gamma \md_{\gamma}\]
(compare to the situation when $\dom\gamma$ is an interval described in Remark \ref{rmk:ambrosio}).
For $x,y \in X$ we define $\Gamma(x,y)$ to be those $\gamma \in \Gamma$ such that $\gamma(\inf\dom\gamma)= x$ and $\gamma(\sup\dom\gamma)= y$.  For $0<\epsilon < 1$ we define
\[\rho_\epsilon(x,y) = \inf_{\gamma\in\Gamma(x,y)} \int_\gamma^* \epsilon.
\]
If $A\subset X$ is Borel and $x,y \in A$ we also define
\[\rho_\epsilon^A(x,y) = \inf_{\substack{\gamma\in\Gamma(x,y),\\\mathrm{im}(\gamma) \subset A\cup \{x,y\}}} \int_\gamma^* \epsilon.\]
Note that both of these are 1-Lipschitz functions and define a metric on $X$.
\end{definition}

We first state the following proposition, whose proof will be given in Section \ref{s:nondiff}.  This is the key proposition that will allow us to conclude that sets that exhibit bad pointwise Lipschitz behavior must have measure zero in RNP-LDS.  The proof will be by contradiction.  Namely, if $S$ has positive measure, then we will glue the functions $f_x$ into one RNP-valued function that is not differentiable on a positive subset of $S$.  This contradicts the hypothesis $(X,d,\mu)$ is RNP-LDS.

\begin{proposition} \label{p:nondiff-construct}
Let $(X,d,\mu)$ be an RNP-LDS and $0<\delta<1$.  Suppose that $S\subset X$ is Borel such that, for every $\epsilon >0$ and almost every $x\in S$ there exists a 1-Lipschitz $f_x^\epsilon\colon X\to \R$ such that
\[\Lip(f_x^\epsilon,z) < \epsilon \quad \mu \text{-a.e. } z\in S\]
and
\[\Lip(f_x^\epsilon,x) > \delta.\]
Then $\mu(S) = 0$.
\end{proposition}

We now demonstrate how the construction given in Proposition \ref{p:nondiff-construct} allows us to join almost every $x$ to every nearby $y$ by curves in $\Gamma$.  This demonstrates why we must assume that the $f^x$ above only behave badly around a single point ($x$).  In turn, this is the reason why we must construct a non-differentiable RNP valued Lipschitz function and hence why we must work inside an RNP-LDS.

\begin{lemma} \label{l:rho-derivative}
Let $(U,\varphi)$ be an $n$-dimensional chart in an RNP-LDS $(X,d,\mu)$ with $\kappa$-universal Alberti representations.  Then for any $x \in X$ and $\epsilon > 0$, the function $f_x^\epsilon(z) = \rho_\epsilon(x,z)$ satisfies
  \begin{align*}
    \Lip(f,z) \leq \frac{\epsilon}{\kappa}, \qquad \mu\mbox{--}a.e. ~z \in U.
  \end{align*}
\end{lemma}

\begin{proof}
For a fixed $x\in X$, let $f=f_x^\epsilon$, and pick a $\gamma\in\Gamma$.  We will show that $|(f\circ\gamma)'(t)| \leq \epsilon$ for $\mathcal L^1$-a.e. $t$.

Indeed, let $t$ be a point of density of $\dom \gamma$ such that $(f \circ \gamma)'(t)$ exists. Let $\delta > 0$ and $R(\delta) > 0$ be so that
  \begin{align}
    |\dom \gamma \cap [t-r,t+r]| > (1-\delta) 2r, \qquad \forall r < R(\delta) \label{e:dom-dense}
  \end{align}
and let $\gamma_t \in \Gamma(x,\gamma(t))$ and $s \in \dom\gamma \cap [t,t+R(\delta))$.  We can define a curve $\gamma_s \in \Gamma(x,\gamma(s))$ by taking $\gamma_t$ and then concatenating $\gamma|_{[t,s]}$ at the end, not allowing for any gaps between the domains of $\gamma_t$ and $\gamma|_{[t,s]}$.  By translating the domain if necessary, we may suppose that $\gamma_s(t) = \gamma(t)$ and that for $r > t$, $\gamma_s(r) = \gamma(r)$.

  We get that
  \begin{align*}
    f(\gamma(s)) = \inf_{\gamma \in \Gamma(x,\gamma(s))} \int_\gamma^* \epsilon \leq \int_{\gamma_s}^* \epsilon = \int_{\gamma_t}^* \epsilon + \int_{\gamma|_{[t,s]}}^* \epsilon.
  \end{align*}
  As $\gamma$ is 1-Lipschitz, we get by \eqref{e:dom-dense} that
  \begin{align*}
    \int_{\gamma|_{[t,s]}}^* \epsilon \leq (s-t) \epsilon + 2 \delta (s-t).
  \end{align*}
  Thus, for each $s \in [t,t+R)$, we get that
  \begin{align*}
    f(\gamma(s)) \leq \int_{\gamma_t}^* \epsilon + (\epsilon + 2 \delta)(s-t).
  \end{align*}
  Taking the infimum of the right hand side over $\gamma_t \in \Gamma(z,\gamma(t))$, we get
  \begin{align*}
    f(\gamma(s)) \leq f(\gamma(t)) + (\epsilon + 2 \delta)(s-t).
  \end{align*}
  Thus, we see that
  \begin{align*}
    \lim_{s \to t^+} \frac{f(\gamma(s)) - f(\gamma(t))}{s-t} \leq \epsilon + 2 \delta.
  \end{align*}
  Taking $\delta \to 0$ gives us
  \begin{align}
    \lim_{s \to t^+} \frac{f(\gamma(s)) - f(\gamma(t))}{s-t} \leq \epsilon. \label{e:fgamma-pos}
  \end{align}
  By reversing the direction of $\gamma$, we can prove that
  \begin{align}
    \lim_{s \to t^-} \frac{f(\gamma(s)) - f(\gamma(t))}{t-s} \leq \epsilon. \label{e:fgamma-neg}
  \end{align}
  Using the fact that $(f \circ \gamma)'(t)$ exists, \eqref{e:fgamma-pos} and \eqref{e:fgamma-neg} imply that $|(f\circ\gamma)'(t)|\leq \epsilon$.

  Let $(\bP_i,\mu_\gamma^i)$ be the $\kappa$-universal Alberti representations of $(U,\varphi)$.  Then, as shown above, for every $1\leq i \leq n$, $|(f\circ \gamma)'(t)| \leq \epsilon$ for $\mathbb P_i$-a.e. $\gamma \in \Gamma$ and $\mu_\gamma^i$-a.e. $t \in\dom\gamma$.  Therefore, for $\mu$-a.e. $z\in U$ we have $\Lip(f,z)\leq \epsilon/\kappa$, as required.
\end{proof}

\begin{proposition}\label{p:somecurves}
Let $(X, d, \mu)$ be a RNP-LDS and $0<\delta<1$.  Then for $\mu$-a.e. $x_0 \in X$, there exists an $\epsilon_0>0$ such that
\begin{align}
  \limsup_{x\to x_0} \frac{\rho_{\epsilon_0}(x,x_0)}{d(x,x_0)} < \delta. \label{e:initial-rho}
\end{align}
\end{proposition}

\begin{proof}
Suppose that the conclusion is false for some $0<\delta<1$.  Then there exists $S\subset X$ of positive measure such that, for any $\epsilon > 0$ and every $x\in S$,
\[\Lip(f_x^\epsilon,x)  > \delta/2,\]
for $f_x^\epsilon(z) = \rho_\epsilon(x,z)$.

Note that $X$ is covered by a countable collection of universal Alberti representations and so we may suppose that there exist a $\kappa>0$ and $U \supset S$ such that $\mu \llcorner U$ has $\kappa$-universal Alberti representations.  Therefore, if we pick $\epsilon_i \searrow 0$, by the previous lemma, for every $i\in \mathbb N$ and $\mu$-a.e. $x\in S$ we have
\[
  \Lip(f_{\epsilon_i}^x,z)\leq \epsilon_i/\kappa, \qquad \mu\text{-a.e.} ~z \in S.
\]
Then $S$ (being an RNP-LDS itself) satisfies the hypotheses of Proposition \ref{p:nondiff-construct} and so we must have $\mu(S)=0$, a contradiction.
\end{proof}

For any $0<\delta<1$ and any suitable $x,x_0$, the previous proposition finds a curve $\gamma\in\Gamma(x,x_0)$ with $\ms \gamma > (1-\delta) \len\gamma$.  That is, a curve fragment with only a very small amount of fragmentation.  However, $\len\gamma \leq \delta d(x,y)/\epsilon_0$ and, of course, $\delta/\epsilon_0$ may not be bounded as $\delta \to 0$.

Fortunately, there is another way to reduce $\delta$ without reapplying the proposition.  Under suitable conditions, we can join the gaps in $\gamma$ by other curve fragments that satisfy the same properties as $\gamma$.  We can then form a new curve $\gamma^*$ that travels mostly along $\gamma$ but travels along these new curves in the gaps of $\gamma$.  Thus $\gamma^*$ has even smaller fragmentation than $\gamma$.  Of course, we must do this in a certain way so that the length of $\gamma^*$ is not too much larger than that of $\gamma$.  By repeating this, we are able to obtain a curve that joins near by points with arbitrarily small fragmentation and controlled length.

In order to do this, we must make sure that the gaps in $\gamma$ begin and end at points that belong to some predetermined ``good set'' (so that the gap can be joined by another good curve fragment).  For simplicity, we will require that \emph{all} of the points in $\gamma$ are good.  We will use the following lemma (at a density point of the ``good set'') to ensure this.

\begin{lemma}\label{l:pokingholes}
For $0< \delta < 1$ suppose that $(X,d,\mu)$ is an RNP-LDS.  Then for $\mu$-a.e. $x \in X$ there exist $D,R,\epsilon_0 >0$ such that, whenever $0<r<R$ and $y\in X$ with $r/2 < d(x,y)<r$ and $A \subset X$ with
\[\frac{\mu(A\cap B(x,\delta r/\epsilon_0))}{\mu(B(x,\delta r/\epsilon_0))} < D,\]
we have
\[\frac{\rho_{\epsilon_0}^{A^c}(x,y)}{d(x,y)} < \delta.\]
\end{lemma}

\begin{proof}
  First notice that, if $d(x,y)<r$ and $\gamma \in \Gamma(x,y)$ satisfies
  \[\int_\gamma^* \epsilon_0 < \delta d(x,y),
  \]
  then $\im \gamma \subset B(x,\delta r/\epsilon_0)$.  Therefore, it suffices to prove the Lemma under the additional assumption that any such set $A$ in the hypothesis of the Lemma in fact satisfies $A \subset B(x, \delta r/\epsilon_0)$.  We suppose that this new conclusion is not true and aim for a contradiction.  Then there exists a Borel $S\subset X$ of positive measure for which the conclusion does not hold for every $D,R,\epsilon >0$ and some $\delta > 0$.

  For a moment, fix $D,R,\epsilon>0$.  Since $\mu$ is pointwise doubling, for $\mu$-a.e. $x\in S$ there exists $C_x \geq 1$ such that, for every $M\geq 1$,
  \begin{equation}\label{e:c-doubling}\limsup_{r\to 0}\frac{\mu(B(x,Mr))}{\mu(B(x,r))} < C_x^{\log_2 M}.\end{equation}
  Then by the Vitali covering theorem, there exists a disjoint covering of $S$ of the form
\[S \subset N \cup \bigcup_{i=1}^\infty B(x_i,r_i)\]
where $\mu(N)=0$ and, for each $i\in\mathbb N$, we have $0<r_i <R$,
\begin{equation}\label{e:high-density-S}\frac{\mu(B(x_i,r_i)\cap S)}{\mu(B(x_i,r_i))}>(1-D)\end{equation}
and there exist $A_i\subset B(x_i,\delta r_i/\epsilon)$ with
\begin{equation}\label{e:disconnected-S}\frac{\mu(B(x_i,\delta r_i/\epsilon)\cap A_i)}{\mu(B(x_i,\delta r_i/\epsilon))} < DC_{x_i}^{-\log_2 \delta/\epsilon}\end{equation}
and $y_i$ with $r_i/2 < d(x_i,y_i)<r_i$ such that
\[\frac{\rho_\epsilon^{A_i^c}(x_i,y_i)}{d(x_i,y_i)} \geq \delta.
\]
In particular,
\begin{align}\mu(B(x_i,\delta r_i/\epsilon)\cap A_i) &\overset{\eqref{e:c-doubling}\wedge \eqref{e:disconnected-S}}{<} D \mu(B(x_i,r_i))\nonumber\\
  &\overset{\eqref{e:high-density-S}}{\leq} 2D \mu(B(x_i,r_i)\cap S),\label{e:low-density-A}
  \end{align}
and for any $x\in B(x_i,r_i)$ we cannot have both
\[\frac{\rho_\epsilon^{A_i^c}(x_i,x)}{d(x_i,x)} < \frac{\delta}{6} \text{ and } \frac{\rho_\epsilon^{A_i^c}(x,y_i)}{d(x,y_i)} < \frac{\delta}{6}.
\]
Notice that this is also true for larger values of $\epsilon$.

Let
\[S_D = S\setminus\bigcup_{i=1}^\infty A_i,\]
and notice that
\begin{align*}
\mu(S_D) &\geq \mu(S) - \sum_{i=1}^\infty \mu((B(x_i,\delta r_i/\epsilon)\cap A_i))\\
&\overset{\eqref{e:low-density-A}}{>} \mu(S)-2D\sum_{i=1}^\infty \mu(B(x_i,r_i)\cap S)\\
&= (1-2D) \mu(S).\end{align*}
Therefore, if we choose $D_i \to 0$ fast enough such that $\prod_i (1-2D_i) > 0$ and arbitrarily choose $R_i,\epsilon_i \to 0$, the set
\[S' = \bigcap_{i=1}^\infty S_{D_i}
\]
satisfies $\mu(S') >0$.  By construction, for any $x \in S'$ and any $\epsilon>0$,
\[\limsup_{y\to x}\frac{\rho_\epsilon^{S'^c}(y,x)}{d(y,x)} \geq \frac{\delta}{6}.
\]

By choosing $\epsilon_n\to 0$ and applying Lemma \ref{l:porosity-subset}, we may replace the $\limsup$ with one over $y\in S'$ for $\mu$-a.e. $x\in S'$.  As $S'$ is a positive measure subset of $X$, it is itself a RNP-LDS.  We then get a contradiction of \eqref{e:initial-rho}.
\end{proof}

We now implement the argument discussed before the previous lemma.

\begin{lemma}\label{l:mainbootstrap}
Fix $0<\delta<1$ and a Borel set $Y\subset X$.  Suppose that all $x_0\in Y$ satisfy the conclusion of Lemma \ref{l:pokingholes} for a fixed $\epsilon_0>0$.  Suppose also that for $\epsilon >0$, all $x_0 \in Y$ satisfy
\begin{equation}\label{e:prebootstrap}
  \limsup_{x\to x_0}\frac{\rho_{\epsilon}(x,x_0)}{\epsilon d(x,x_0)} \leq \frac{\delta}{\epsilon_0}.\end{equation}
Then for $\mu$-a.e. $x_0\in Y$,
\begin{equation}\label{e:bootstrap}
  \limsup_{x\to x_0}\frac{\rho_{\delta\epsilon}(x,x_0)}{\delta\epsilon d(x,x_0)} \leq \frac{\delta}{\epsilon_0}.
\end{equation}
\end{lemma}

\begin{proof}
Let $M:=\delta/\epsilon_0$ and $\eta \in (0,1)$.  For $\alpha>0$ let $S'_\alpha$ be the Borel set of those $x_0\in Y$ for which
\[\sup_{z \in B(x_0,\alpha)\setminus\{x_0\}} \frac{\rho_{\epsilon}(x_0,z)}{\epsilon d(x_0,z)} < M + \eta.\]
For $\mu$-a.e. $z\in S'_\alpha$ we let $D(z), R(z)>0$ be given by the hypothesis on $Y$.  Also, by the Lebesgue density theorem, for $\mu$-a.e. $z\in S'_\alpha$ there exists an $R'(z)$ such that, for all $0<r<R'(z)$,
\[\frac{\mu(B(z,r)\cap S'_\alpha)}{\mu(B(z,r))} > 1-D.\]
We let $S_\alpha$ be the Borel set of those $z\in S'_\alpha$ for which $\alpha<D(z),R(z),R'(z)$, so that $S_\alpha$ increases to a set of full measure in $Y$ as $\alpha \to 0$.

  Fix an $\alpha>0$, let $0<\eta<1$ and fix $x\in S_\alpha$ and $y \in X$ with $0<d(x,y)<\alpha$.  Then, by the conclusion of Lemma \ref{l:pokingholes}, there exists a $\gamma \in \Gamma(x,y)$ with $\mathrm{im}(\gamma) \subset S_\alpha' \cup \{x,y\}$ and
  \begin{align}
    \int_\gamma^* \epsilon_0 = \epsilon_0\ms \gamma + (\len\gamma-\ms\gamma)\leq \delta d(x,y) < \alpha. \label{e:epsilon0-*integral}
  \end{align}
Let $[a,b]$ be the smallest interval that contains $\dom\gamma$ and write $(a,b) \setminus \dom\gamma = \cup_1^\infty (a_i,b_i)$.  Let $N\in\N$ be such that $\sum_{i>N}b_i-a_i < \eta$.

By construction, each $\gamma(a_i) \in S'_\alpha$ and we get from the 1-Lipschitzness of $\gamma$ that
\[d(\gamma(a_i), \gamma(b_i)) \leq |b_i -a_i| \overset{\eqref{e:epsilon0-*integral}}{<} \alpha.\]
Therefore, by \ref{e:prebootstrap},
\[\rho_{\epsilon}(\gamma(a_i),\gamma(b_i)) < \epsilon (M+ \eta) d(\gamma(a_i),\gamma(b_i))\]
and so there exists $\gamma_i \in \Gamma(\gamma(a_i),\gamma(b_i))$ with
\[\int_{\gamma_i}^* \epsilon \leq \epsilon (M+\eta) d(\gamma(a_i),\gamma(b_i)).\]
We define a new curve $\gamma^* \in \Gamma(x,y)$ by traveling along $\gamma_i$ instead of $\gamma|_{(a_i,b_i)}$ for each $i \leq N$ (and expanding the domain of $\gamma$ appropriately).  Then,
\begin{align*}
\frac{1}{\delta\epsilon}\int_{\gamma^*}^* \delta\epsilon &= \ms \gamma^* + \frac{1}{\delta\epsilon}(\len\gamma^*-\ms\gamma^*)\\
&\leq \ms\gamma + \sum_{i=1}^N \ms\gamma_i+ \frac{1}{\delta\epsilon}\left( \sum_{i=1}^N \len\gamma_i - \ms\gamma_i\right) + \frac{1}{\delta\epsilon} \left(\mathcal \sum_{i>N} b_i-a_i \right) \\
&\leq \ms\gamma +\frac{1}{\delta\epsilon}\sum_{i=1}^N \left(\epsilon \ms\gamma_i + \len\gamma_i-\ms\gamma_i\right)+ \frac{\eta}{\delta\epsilon}\\
&\leq \ms\gamma +\frac{1}{\delta\epsilon}\sum_{i=1}^N \int_{\gamma_i}^* \epsilon+ \frac{\eta}{\delta\epsilon}\\
&\leq \ms\gamma + \frac{M + \eta}{\delta}\sum_{i=1}^N d(\gamma(a_i),\gamma(b_i)) + \frac{\eta}{\delta\epsilon}\\
&\leq \ms\gamma + \frac{1}{\epsilon_0} (\len\gamma-\ms\gamma)+ \frac{\eta}{\delta\epsilon}(1 + b-a)\\
&= \frac{1}{\epsilon_0} \int_\gamma^*\epsilon_0+ \frac{\eta}{\delta\epsilon} ( 1 + b-a)\\
&\overset{\eqref{e:epsilon0-*integral}}{\leq} \frac{\delta}{\epsilon_0}d(x,y) + \frac{\eta}{\delta\epsilon} (1 + b-a)
\end{align*}
This is true for every $0<\eta<1$ and so
\[\frac{\rho_{\delta\epsilon}(x,y)}{\delta\epsilon d(x,y)} \leq \frac{\delta}{\epsilon_0}=M.\]

This shows that every $x\in S_\alpha$ satisfies the conclusion of the lemma.  The proof is completed by taking a countable union over $\alpha_n \to 0$.
\end{proof}

By repeatedly applying the previous lemma, we show that an RNP-LDS satisfies a fragmented version of the quasiconvex condition.

\begin{lemma}\label{l:easybootstrap}
Let $(X,d,\mu)$ be an RNP-LDS.  Then for $\mu$-a.e. $x_0\in X$,
\[\sup_{0<\epsilon <1}\limsup_{x\to x_0}\frac{\rho_{\epsilon}(x,x_0)}{\epsilon d(x,x_0)} < \infty.\]
\end{lemma}

\begin{proof}
First observe that the supremum in the conclusion of the lemma is in fact a limit as $\epsilon \to 0$.  Therefore it suffices to prove that the limit is finite for countable $\epsilon_n\to 0$.

Fix $0<\delta<1$.  Then by Lemma \ref{l:pokingholes}, there exists a countable number of Borel sets $Y_i$ that cover almost all of $X$ such that every $x_0\in Y_i$ satisfies the conclusion of the lemma for $\epsilon_0=1/i$.  It suffices to prove the result for $\mu$-a.e. $x_0 \in Y_i$.  From now on we fix this value of $\epsilon_0$ and write $Y$ for $Y_i$.

Note that, by putting $A=\emptyset$ into Lemma \ref{l:pokingholes}, we immediately have
\[\limsup_{x\to x_0} \frac{\rho_{\epsilon_0}(x,x_0)}{d(x,x_0)} < \delta\]
for $\mu$-a.e. $x_0\in Y$.  Therefore \eqref{e:prebootstrap} is true for $\epsilon=\epsilon_0$ and so, by Lemma \ref{l:mainbootstrap}, \eqref{e:bootstrap} is true for $\epsilon=\epsilon_0$.  Thus \eqref{e:prebootstrap} is true for $\epsilon=\delta \epsilon_0$ and so we may apply Lemma \ref{l:mainbootstrap} again to deduce \eqref{e:bootstrap} for $\epsilon = \delta \epsilon_0$ (with the understanding that all of these assertions are true for $\mu$-a.e. $x_0\in Y$).  We continue repeatedly applying the lemma for $\epsilon_n= \delta^n \epsilon_0 \to 0$ and thus completing the proof.
\end{proof}

Finally, we see that we can also obtain a conclusion similar to that of Lemma \ref{l:pokingholes} but for all $\epsilon > 0$.

\begin{theorem}\label{t:connectingpoints}
  Let $(X,d,\mu)$ be an RNP-LDS.  Then for $\mu$-a.e. $x_0\in X$ there exist $\lambda >0$ and, for every $\epsilon>0$, a $D,R>0$ such that, whenever $y\in X$ with $0<r/2 <d(x,y)<r<R$ and $A \subset X$ with
  \[\frac{\mu(B(x,\lambda r) \cap A)}{\mu(B(x,\lambda r))} < D,\]
  we have
  \[\frac{\rho_{\epsilon}^{A^c}(x,y)}{\epsilon d(x,y)} < \lambda.\]
\end{theorem}

\begin{proof}
  Suppose that the conclusion is false.  Then there exists an $S\subset X$ of positive measure such that, for every $\lambda \in \mathbb N$ and $x_0\in S$, there exists an $\epsilon_\lambda(x_0) >0$ for which the conclusion is false for all $D,R>0$.

  Analogously to the proof of Lemma \ref{l:pokingholes}, choose $D_m>0$ with $\prod (1-D_m) > 0$ and for each $m\in \mathbb N$ construct a set $S_m\subset S$ using the Vitali covering theorem with $\mu(S_m)\geq (1-D_m)\mu(S)$ such that, for each $x_0\in S_m$ there exists a $y\in X$ with $d(y,x_0)<1/m$ and
  \[\frac{\rho_{\epsilon_m}^{S_m}(x_0,y)}{\epsilon_m d(x_0,y)} \geq m.\]
  Notice that this is also true for all $\epsilon < \epsilon_m(x_0)$.

  Therefore $S'=\cap S_m$ has positive measure but
  \[\sup_{0<\epsilon<1}\limsup_{x\to x_0}\frac{\rho_\epsilon^{S'^c}(x,x_0)}{\epsilon d(x,x_0)} = \infty\]
  for $\mu$-a.e. $x_0\in S'$.  By Lemma \ref{l:porosity-subset} we may take the $\limsup$ over $x\in S'$, contradicting Lemma \ref{l:easybootstrap} and the fact that $S'$ is an RNP-LDS.
\end{proof}

\section{Non-homogeneous Poincar\'e inequalities}\label{s:poincare}

We would like to show that RNP-Lipschitz differentiability implies an asymptotic and non-homogeneous form of the Poincar\'e inequality.  However, to discuss the Poincar\'e inequality, we need a notion of upper gradients.  The original definition does not work for us because, as mentioned before, $X$ may lack any rectifiable curves.  Instead, we use the *-integral to define upper gradients for 1-Lipschitz functions via curve fragments.  Specifically, we say that $\rho : X \to [0,1]$ is a *-upper gradient of the 1-Lipschitz function $u : X \to \R$ if for all $x,y \in X$ and $\gamma \in \Gamma(x,y)$ we have
\begin{align*}
  |u(x) - u(y)| \leq \int_\gamma^* \rho.
\end{align*}

Note that every 1-Lipschitz function $u$ has an *-upper gradient by taking $\rho \equiv 1$.  Here, it is crucial that $u$ is 1-Lipschitz and not continuous or even arbitrarily $L$-Lipschitz for $L \geq 1$.  This is because the *-integral assigns mass at a density of 1 to the gaps in $\gamma$.  Otherwise, we are not guaranteed that *-upper gradients exist.  Of course in the definition of the *-integral, the use of 1 was arbitrary.

More generally, we have the following.
\begin{lemma} \label{l:Lip-ug}
  Let $f : X \to \R$ be 1-Lipschitz.  Then $\Lip f$ is a *-upper gradient of $f$.
\end{lemma}
The proof is a straightforward modification of Proposition 1.11 of \cite{cheeger}

We first show that RNP-LDS satisfy a pointwise Poincar\'e style inequality for indicator functions of measurable subsets (see \cite[Theorem 9.5]{heinonen} for analogous statements in PI spaces).  As it relates connectivity of points to densities of sets, this can be thought of as a ``geometric'' Poincar\'e style inequality.  We require the following quasiconvexity type condition on $\Gamma$: For $x,y \in X$ and $\lambda \geq 1$, we let
\begin{align*}
  \Gamma(x,y;\lambda) = \{\gamma \in \Gamma(x,y) : \len(\gamma) \leq \lambda d(x,y) \}.
\end{align*}

\begin{lemma} \label{l:geom-PI}
  Let $(X,d,\mu)$ be a RNP-LDS.  Then for $\mu$-a.e. $x \in X$, there exist $\tilde{\lambda}_x \geq 1$ and functions $\tilde{o}_x,\tilde{\zeta}_x$ so that for every measurable $A \subseteq X$ and $y \in X$ we have
  \begin{align}
    \inf_{\gamma \in \Gamma(x,y;2\tilde{\lambda}_x)} \int_\gamma^* {\bf 1}_A  \leq d(x,y)\tilde{\zeta}_x \left( \fint_{B(x,\tilde{\lambda}_x d(x,y))} {\bf 1}_A ~d\mu \right) + \tilde{o}_x(d(x,y)). \label{e:geom-PI}
  \end{align}
  Here the $\tilde{o}_x : [0,\infty) \to \R$ and $\tilde{\zeta}_x : [0,1] \to \R$ are increasing and satisfy $\tilde{o}_x(0) = \tilde{\zeta}_x(0) = 0$ and $\tilde o_x(r)/r \to 0$ as $r\to 0$.
\end{lemma}

\begin{proof}
  By Theorem \ref{t:connectingpoints}, for $\mu$-a.e. $x \in X$ (which we fix for the rest of the proof) there exists some $\lambda \geq 1$ and moduli $D(\epsilon), R(\epsilon) > 0$ so that if $r= |x-y| \leq R(\epsilon)$ and $A \subseteq X$ is Borel such that $\mu(A \cap B(x,\lambda r)) < D(\epsilon) \mu(B(x,\lambda r))$, then there exists some $\gamma \in \Gamma(x,y)$ so that $\text{im}(\gamma) \subset A^c \cup \{x,y\}$ and
  \begin{align}
    \int_\gamma^* \epsilon \leq \lambda \epsilon d(x,y). \label{e:star-ineq-1}
  \end{align}
  Note that as $\text{im}(\gamma) \subset A^c \cup \{x,y\}$, we then have that
  \begin{align}
    \int_\gamma^* {\bf 1}_A \leq \int_\gamma^* \epsilon. \label{e:star-ineq-2}
  \end{align}
  We have that
  \begin{align*}
    (\len \gamma - \ms \gamma) + \epsilon \ms \gamma = \int_\gamma^* \epsilon \overset{\eqref{e:star-ineq-1}}{\leq} \lambda \epsilon d(x,y).
  \end{align*}
  This implies that $\ms \gamma \leq \lambda d(x,y)$ and $\len \gamma - \ms \gamma \leq \lambda \epsilon d(x,y)$.  Thus, $\len \gamma \leq (\lambda + \lambda \epsilon) d(x,y) \leq 2 \lambda d(x,y)$ and so $\gamma \in \Gamma(x,y;2\lambda)$.  Thus, we have for $A,x,y$ satisfying the above conditions that
  \begin{align}
    \inf_{\gamma \in \Gamma(x,y;2\lambda)} \int_\gamma^* {\bf 1}_A \overset{\eqref{e:star-ineq-1} \wedge \eqref{e:star-ineq-2}}{\leq} \lambda \epsilon d(x,y). \label{e:star-ineq-3}
  \end{align}

  By taking suboptimal $D$ and $R$ if necessary, we may also assume that $D$ and $R$ are both continuous, strictly increasing, and satisfy $R(0) = D(0) = 0$.  Define the functions
  \begin{align}
    \tilde{\zeta}(t) &:= \lambda D^{-1}(t), \notag \\
    \tilde{o}(t) &:= \lambda R^{-1}(2t)t. \label{e:moduli-set}
  \end{align}
  It follows from the properties of $D$ and $R$ that $\tilde{\zeta}$ and $\tilde{o}$ are continuous, increasing, $\tilde{\zeta}(0) = \tilde{o}(0) = 0$ and $\tilde o(r)/r \to 0$ as $r\to 0$.

  Let $r = d(x,y)$ and $\epsilon_0 = R^{-1}(2r)$.  We want to verify \eqref{e:geom-PI}.  If $A$ is a Borel set such that $\mu(A \cap B(x,\lambda r)) < D(\epsilon_0) \mu(B(x,\lambda r))$, then as $r < R(\epsilon_0)$, we have that
  \begin{align*}
    \inf_{\gamma \in \Gamma(x,y;2\lambda)} \int_\gamma^* {\bf 1}_A  \overset{\eqref{e:star-ineq-3}}{\leq} \lambda \epsilon_0 d(x,y) = \tilde{o}(d(x,y)).
  \end{align*}
  Now suppose $A$ is a Borel set such that $\mu(A \cap B(x,\lambda r)) = D(\epsilon) \mu(B(x,\lambda r))$ for some $\epsilon > \epsilon_0$.  
  Then as $r < R(\epsilon)$, we get that
  \begin{align*}
    \inf_{\gamma \in \Gamma(x,y;2\lambda)} \int_\gamma^* {\bf 1}_A  \overset{\eqref{e:star-ineq-3}}{\leq} d(x,y)\lambda \epsilon = d(x,y)\zeta_{1}(D(\epsilon)) = d(x,y)\zeta_{1}\left( \fint_{B(x,\lambda d(x,y))} {\bf 1}_A ~d\mu\right).
  \end{align*}
  Thus, we see that regardless of $A$, \eqref{e:geom-PI} is satisfied.
\end{proof}

We now show that the pointwise moduli and $\lambda_x$ of Lemma \ref{l:geom-PI} can be chosen from only a countable number of such moduli and $\lambda_x$.

\begin{lemma} \label{l:unif-geom-PI}
  Let $(X,d,\mu)$ be a RNP-LDS.  Then there is a countable Borel decomposition $X = \bigcup_i U_i$ so that for each $i$ there exist $\tilde{\lambda}_i \geq 1, \tilde{\zeta}_i : [0,1] \to \R$, and $\tilde{o}_i : [0,\infty) \to \R$ so that for $\mu$-a.e. $x \in U_i$ and every $y \in X, A \subseteq X$, we have
  \begin{align}
    \inf_{\gamma \in \Gamma(x,y;2\tilde{\lambda}_i)} \int_\gamma^* {\bf 1}_A  \leq d(x,y)\tilde{\zeta}_i \left( \fint_{B(x,\tilde{\lambda}_i d(x,y))} {\bf 1}_A ~d\mu \right) + \tilde{o}_i(d(x,y)). \label{e:unif-geom-PI}
  \end{align}
  Here, $\tilde{\zeta}_i$ and $\tilde{o}_i$ are continuous, increasing, and satisfy $\tilde{\zeta}_i(0) = \tilde{o}_i(0) = 0$ and $\tilde o_i(r)/r \to 0$ as $r\to 0$.
\end{lemma}

Note that the integral on the right hand side of \eqref{e:unif-geom-PI} is performed over a ball in the whole of $X$, not just $U_i$.

\begin{proof}
We first do a few reductions.  By the proof of Lemma \ref{l:geom-PI} (using \eqref{e:moduli-set} for example), it suffices to show that there is a countable Borel decomposition $X = \bigcup_i U_i$ so that for each $i$ there exists $\lambda_i \geq 1$ and $D_i(\epsilon),R_i(\epsilon) > 0$ so that the conclusion of Theorem \ref{t:connectingpoints} holds with these quantities uniformly for $\mu$-a.e. $x \in U_i$.  Further, it further suffices to demonstrate the existence of $D_i(\epsilon_k)$ and $R_i(\epsilon_k)$ for a countable sequence $\epsilon_k \to 0$ (for each $i$).

Let $D,R,\epsilon \in (0,1)$, $\lambda \geq 1$, and let $S(\lambda,D,R,\epsilon)$ be the set of $x \in X$ for which the following holds: whenever $r < R$ and $y \in X$ with $d(x,y) \in (r/2,r)$ and $A \subseteq X$ is Borel such that
  \begin{align}\label{e:small-density}
    \frac{\mu(B(x, \lambda r) \cap A)}{\mu(B(x, \lambda r))} < D,
  \end{align}
then $\rho_\epsilon^{A^c}(x,y) \leq \lambda \epsilon d(x,y)$.  Then each $S(\lambda,D,R,\epsilon)$ is closed.  We will prove this for the two cases depending on whether balls are assumed to be open or closed.

Indeed, first suppose that balls are open.  For fixed $D,R,\epsilon,\lambda$ suppose that $x_n \in S(\lambda, D,R,\epsilon)$ with $d(x_n,x) \searrow 0$.  Let $A\subset X$ satisfy \eqref{e:small-density} and $y \in X$ be so that $d(y,x) \in (r/2,r)$.  If we set $r_n = r - (1 + \frac{1}{n}) d(x_n,x)$ we get that $\overline{B(x_n, \lambda r_n)}$ is an increasing sequence of closed sets filling up $B(x,\lambda r)$ and so by inner regularity we have
\begin{align}
  \mu(B(x_n,\lambda r_n)) \nearrow \mu(B(x,\lambda r)). \label{e:inner-reg-conv}
\end{align}
As $\mu(B(x_n,\lambda r_n) \cap A) \leq \mu(B(x,\lambda r) \cap A)$, we get from \eqref{e:small-density} and \eqref{e:inner-reg-conv} that there exists $N_1 \geq 1$ so that $\mu(B(x,\lambda r_n) \cap A) < D \mu(B(x_n,\lambda r_n)$ for all $n \geq N_1$.  As $x_n \to x$, $r_n \to r$, and $d(x,y) \in (r/2,r)$, there then exist $N_2 \geq N_1$ so that $d(y,x_n) \in (r_n/2,r_n)$ for all $n \geq N_2$.  Thus, as $x_n \in S(\lambda, D,R, \epsilon)$, we have $\rho_\epsilon^{A^c}(x_n,y) \leq \lambda \epsilon d(x_n,y)$ for all $n \geq N_2$.

Fix an $n \in \N$.  For any curve $\gamma \in \Gamma(x_n,y)$ so that $\text{im}(\gamma) \subset A^c \cup \{x,y\}$ (and we suppose $\inf \dom(\gamma) = 0$), we have that the curve
\begin{align*}
  \tilde{\gamma} : \{-d(x,x_n)\} \cup (\dom(\gamma) \backslash [0,d(x,x_n))) &\to X \\
  t &\mapsto \begin{cases}
    x & t = -d(x,x_n), \\
    \gamma(t) & t \neq -d(x,x_n),
  \end{cases}
\end{align*}
is in $\Gamma(x,y)$ and has $\text{im}(\tilde{\gamma}) \subset A^c \cup \{x,y\}$.  Thus,
\begin{align*}
  \rho_\epsilon^{A^c}(x,y) \leq \inf_{\gamma \in \Gamma(x_n,y)} \int_{\tilde{\gamma}}^* \epsilon \leq 2 d(x,x_n) + \rho_\epsilon^{A^c}(x_n,y) \leq 2 d(x,x_n) + \lambda \epsilon (d(x,y) + d(x_n,x)).
\end{align*}
Taking $n \to \infty$, we get $\rho_\epsilon^{A^c}(x,y) \leq \lambda \epsilon d(x,y)$, as needed.

If balls are closed, then we instead take $B(x,\lambda_n r_n) \searrow B(x,\lambda r)$ (and hence $B(x,\lambda_n r_n)\cap A \searrow B(x,\lambda r)\cap A$) by inflating $r_n$ and then apply outer regularity to get the same statement.

By Theorem \ref{t:connectingpoints}, for each $\epsilon \in (0,1)$, $S(\lambda,D,R,\epsilon)$ converges to a set of full measure as $D,R\to 0$ and $\lambda \to \infty$.  For any $\delta > 0$, by taking a sufficiently large $\lambda$ and for each $\epsilon$ suitably small $D$ and $R$ and then intersecting over a countable choice of $\epsilon_k \to 0$, we can obtain a subset $S \subset X$ satisfying the property with $\mu(S) \geq \mu(X) - \delta$ (recall that we are always under the assumption that $X$ has finite measure).  To finish, one then takes a countable sequence $\delta_k \to 0$ and exhaust $S$ by such constructed subsets corresponding to each $\delta_k$.
\end{proof}

We now show that the previous geometric Poincar\'e style inequality can be upgraded to handle measurable functions.

\begin{lemma} \label{l:funct-PI}
  Let $U_i$ be any of the Borel pieces from the decomposition of Lemma \ref{l:unif-geom-PI}.  Then there exist $\hat{\lambda}_i \geq 1$ and functions $o_i,\hat{\zeta}_i$ so that for $\mu$-a.e. $x \in U_i$, every measurable $\rho : X \to [0,1]$ and $y \in X$ we have
  \begin{align}
    \inf_{\gamma \in \Gamma(x,y;2\hat{\lambda}_i)} \int_\gamma^* \rho  \leq d(x,y) \hat{\zeta}_i \left( \fint_{B(x,\hat{\lambda}_i d(x,y))} \rho ~d\mu \right) + \hat{o}_i(d(x,y)). \label{e:funct-PI}
  \end{align}
  Here the $\hat{o}_i : [0,\infty) \to \R$ and the $\hat{\zeta}_i : [0,1] \to \R$ are continuous, increasing, and satisfy $\hat{o}_i(0) = \hat{\zeta}_i(0) = 0$ and $\hat o_i(r)/r \to 0$ as $r\to 0$.
\end{lemma}

\begin{proof}
  Let $\lambda,o,\zeta$ be the constant and moduli guaranteed by Lemma \ref{l:unif-geom-PI} so that \eqref{e:unif-geom-PI} holds (we omit the tildes and $i$ subscript for notational simplicity).  Now let $y \in X$ and $\rho : X \to [0,1]$ be measurable.  For $t \in [0,1]$, let $A_t = \rho^{-1}([0,t])$.  We then have that
  \begin{align}
    \frac{\mu(A_t^c \cap B(x,\lambda d(x,y)))}{B(x, \lambda d(x,y))} \leq \frac{1}{t} \fint_{B(x, \lambda d(x,y))} \rho ~d\mu. \label{e:markov-ineq}
  \end{align}
  We then have for any $t$ that
  \begin{align*}
    \inf_{\gamma \in \Gamma(x,y;2\lambda)} \int_\gamma^* \rho  &\leq \inf_{\gamma \in \Gamma(x,y;2\lambda)} \int_\gamma^* t {\bf 1}_{A_t} + {\bf 1}_{A_t^c} \\
    &\leq 2\lambda t d(x,y)+ d(x,y)\inf_{\gamma \in \Gamma(x,y;2\lambda)} \int_\gamma^* {\bf 1}_{A_t^c}  \\
    &\overset{\eqref{e:unif-geom-PI}}{\leq} 2\lambda t d(x,y)+ d(x,y)\zeta\left( \fint_{B(x, \lambda d(x,y))} {\bf 1}_{A_t^c} ~d\mu \right) + o(d(x,y)) \\
    &\overset{\eqref{e:markov-ineq}}{\leq} 2\lambda td(x,y) + d(x,y)\zeta\left( \frac{1}{t} \fint_{B(x, \lambda d(x,y))} \rho ~d\mu \right) + o(d(x,y)).
  \end{align*}
  Thus, if we set $\hat{\lambda}_i = \lambda$, $\hat{o}_i = o$, and
  \begin{align*}
    \hat{\zeta}_i(r) = \inf_{t \in (0,1)} 2\lambda t + \zeta(r/t),
  \end{align*}
  we get \eqref{e:funct-PI}.  That $\hat{\zeta}_i$ is increasing and satisfies $\hat{\zeta}_i(0) = 0$ is an easy exercise.
\end{proof}

Given a measurable $\rho : X \to [0,1]$, we define
\begin{align*}
  \tilde{F}_\rho(x,y;\lambda) := \inf_{\gamma \in \Gamma(x,y;\lambda)} \int_\gamma^* \rho.
\end{align*}
This should be compared to the definition in \cite{cheeger} that does not use the *-integral.  The next two lemmas prove an asymptotic nonhomogeneous quasiconvex analogue of the segment inequality of \cite{cheeger-colding} for RNP-LDS.  As in the Riemannian case, this inequality will be the means by which we prove our asymptotic nonhomogeneous Poincar\'e inequality.

\begin{lemma} \label{l:poincare-unif-unif}
  Let $U_i$ be any of the Borel pieces from the decomposition of Lemma \ref{l:funct-PI}.  Then there exists some $\lambda_i \geq 1$ and $\zeta_i : [0,1] \to \R$ so that $\mu$-a.e. $z \in U_i$, there exists some $o_z : [0,\infty) \to \R$ so that for every measurable $\rho : X \to [0,1]$, we have
  \begin{align}
    \fint_{B(z,r)} \fint_{B(z,r)} \tilde{F}_\rho(x,y;\lambda_i) ~d\mu(x) ~d\mu(y) \leq r \zeta_i \left( \fint_{B(z,(\lambda_i + 1)r)} \rho ~d\mu \right) + o_z(r). \label{e:unif-funct-PI}
  \end{align}
  Here $o_z$ satisfies
  \begin{align}
    \lim_{t \to 0} \frac{o_z(t)}{t} = 0, \label{e:to-o}
  \end{align}
  and $\zeta_i$ is continuous increasing and satisfies $\zeta_i(0) = 0$.
\end{lemma}

As before, the balls $B(z,r)$ and $B(z,(2\lambda+1)r)$ are in $X$ and not restricted to $S$.

\begin{proof}
  We may suppose $U_i$ has positive measure and $z$ is a density point of $U_i$.  Let $\lambda,\zeta,o$ be the constants and moduli guaranteed for $U_i$ by Lemma \ref{l:funct-PI} so that \eqref{e:funct-PI} holds (we omit the hat and $i$ subscript for notational simplicity).  We may also suppose that $U_i$ is uniformly $(C,R)$-doubling.  We define
  \begin{align*}
    o_z(t) :=
      \begin{cases}
        2t \cdot \left[  o(2t) + \left( \frac{\mu(B(z,t) \backslash U_i)}{B(z,t)} \right)^2 \right], & t < \frac{R}{2\lambda+1}, \\
        2t & t \geq \frac{R}{2\lambda + 1}.
      \end{cases}
  \end{align*}
  It follows from the fact that $o(0) = 0$ and the fact that $z$ is a density point of $S$ that \eqref{e:to-o} is satisfied.

  We may suppose without loss of generality that $\zeta$ is concave.  Define
  \begin{align*}
    \tilde{\zeta}(t) := \zeta(t^{1/s}),
  \end{align*}
  where $s$ is the constant from Lemma \ref{l:doubling-balls}.  Note then that for any $\eta \in (0,1)$, we have that
  \begin{align*}
    \tilde{\zeta}(\eta^s t) = \zeta(\eta t^{1/s}) \geq \eta \zeta(t^{1/s}) = \eta \tilde{\zeta}(t), \qquad \forall t > 0.
  \end{align*}
  In the above inequality, we used concavity of $\zeta$ along with $\zeta(0) = 0$.  Thus, by Lemma \ref{l:doubling-balls}, we get that if $(\lambda+1)r < R$, $x \in U_i$, and $B(x,r') \subseteq B(z,(\lambda+1)r)$ then for all $t>0$
  \begin{multline}
    r' \tilde{\zeta}\left( \frac{1}{\mu(B(x,r'))} t \right) \leq 4(2\lambda+1)r \tilde{\zeta} \left( \frac{1}{\mu(B(x,r'))} \left(\frac{r'}{4(2\lambda+1)r} \right)^s t \right) \\
    \leq 4(2\lambda+1)r \tilde{\zeta}\left( \frac{1}{\mu(B(z,(2\lambda +1)r))} t \right). \label{e:tilde-zeta-ineq}
  \end{multline}

  Let $r \in (0,R/(2\lambda+1))$, $A_1 = (B(x,r) \backslash U_i) \times (B(x,r) \backslash U_i)$, and $A_2 = (B(x,r) \times B(x,r)) \backslash A_1$.  Let $(x,y) \in A_2$ and suppose $x \in U_i$.  Then $B(x,\lambda d(x,y)) \subseteq B(z,(2\lambda+1)r)$ and so we get that
  \begin{multline}
    \tilde{F}_\rho(x,y;2\lambda) \overset{\eqref{e:funct-PI}}{\leq} d(x,y) \tilde{\zeta} \left( \fint_{B(x,\lambda d(x,y))} \rho ~d\mu \right) + d(x,y) o(d(x,y)) \\
    \overset{\eqref{e:tilde-zeta-ineq}}{\leq} 4(2\lambda +1)r \tilde{\zeta} \left( \fint_{B(z,(2\lambda +1)r))} \rho ~d\mu \right) + 2r o(2r). \label{e:x-in}
  \end{multline}
  The same inequality holds if $y \in U_i$.  We always have the trivial inequality
  \begin{align}
    \tilde{F}_\rho(x,y;2\lambda) \leq d(x,y), \qquad \forall x,y \in X. \label{e:neither}
  \end{align}

  We now verify \eqref{e:unif-funct-PI}.  By the definition of $o_z$ and the trivial bound \eqref{e:neither}, we have that \eqref{e:unif-funct-PI} is satisfied when $r \geq R/(2\lambda+1)$.  Thus, we may assume that $r < R/(2\lambda+1)$.  Then
  \begin{align*}
    \fint_{B(z,r)} &\fint_{B(z,r)} \tilde{F}_\rho(x,y;2\lambda) ~d\mu(x) ~d\mu(y) \\
    &= \frac{1}{\mu(B(z,r))^2} \left[ \int_{A_2} \tilde{F}_\rho(x,y;2\lambda) d(\mu \times \mu) + \int_{A_1} \tilde{F}_\rho(x,y;2\lambda) d(\mu \times \mu) \right] \\
    &\overset{\eqref{e:x-in} \wedge \eqref{e:neither}}{\leq} 4(2\lambda+1)r \tilde{\zeta} \left( \fint_{B(z,(\lambda+1)r))} \rho ~d\mu \right) + 2r o(2r) + 2r \left( \frac{\mu(B(z,r) \backslash U_i)}{B(z,r)} \right)^2
  \end{align*}
  Taking $\zeta_i(t) = 4(2\lambda+1) \tilde{\zeta}(t)$ and $\lambda_i = 2\lambda$ completes the proof.
\end{proof}

\begin{theorem} \label{t:ae-poincare}
  Let $(X,d,\mu)$ be a RNP-Lipschitz differentiability space.  Then there exist a countable Borel decomposition $X = \bigcup_i U_i$, $\lambda_i \geq 1$, increasing $\zeta_i : [0,1] \to \R$ and $o_i : [0,\infty)\to R$ such that for almost every $z \in U_i$ and every measurable $\rho : X \to [0,1]$ we have
  \begin{align}
    \fint_{B(z,r)} \fint_{B(z,r)} \tilde{F}_\rho(x,y;\lambda_i) ~d\mu(x) ~d\mu(y) \leq r \zeta_i \left( \fint_{B(z,(\lambda_i+1)r)} \rho ~d\mu \right) + o_i(r). \label{e:theorem-funct-PI}
  \end{align}
  Here $o_i$ and $\zeta_i$ satisfy
  \begin{align*}
    \lim_{t \to 0} \zeta_i(t) = 0, \text{ and } \lim_{t \to 0} \frac{o_i(t)}{t} = 0,
  \end{align*}
\end{theorem}

\begin{proof}
  Lemmas \ref{l:unif-geom-PI}, \ref{l:funct-PI}, and \ref{l:poincare-unif-unif} almost give the conclusion needed except that the $o$ error term may still depend on the point in $U_i$.  However, we can take a further countable Borel decomposition of each $U_i$ into pieces $V_{i,j}$ so that \eqref{e:theorem-funct-PI} holds for a single $o_{i,j}$ term on the whole of $V_{i,j}$.  The union of countable decompositions $\{V_{i,j}\}_{i,j}$ for each chart (of which there are a countable number) is the one that satisfies the theorem.
\end{proof}

Using the previous theorem, we can produce the more standard form of a Poincar\'e inequality involving the average of a Lipschitz function.

\begin{theorem} \label{t:true-PI}
  Let $(X,d,\mu)$ be a RNP-Lipschitz differentiability space.  Then there exist a countable Borel decomposition $X = \bigcup_i U_i$, $\lambda_i \geq 1$, increasing $\zeta_i : [0,1] \to \R$ and $o_i : [0,\infty) \to \R$ such that for almost every $z \in U_i$ and every 1-Lipschitz function $u : X \to \R$ with *-upper gradient $\rho : X \to [0,1]$, we have
  \begin{align}
    \fint_{B(z,r)} | u - u_{B(z,r)} | ~d\mu \leq r \zeta_i \left( \fint_{B(z,\lambda_i r)} \rho ~d\mu \right) + o_i(r). \label{e:true-PI}
  \end{align}
  Here $o_i$ and $\zeta_i$ satisfy
  \begin{align*}
    \lim_{t \to 0} \frac{o_i(t)}{t} = 0 \text{ and } \lim_{t \to 0} \zeta_i(t) = 0.
  \end{align*}
\end{theorem}

Again, the balls for all the integrals of \eqref{e:true-PI} are taken in $X$ not $U_i$.

\begin{proof}
  By Theorem \ref{t:ae-poincare}, we get a countable decomposition $U_i$ so that for each $i$ we have some $\lambda_i \geq 1$ so that for almost every $z \in U_i$, we have
  \begin{align}
    \fint_{B(z,r)} \fint_{B(z,r)} \tilde{F}_\rho(x,y;\lambda_i) ~d\mu(x) ~d\mu(y) \leq r \zeta_i \left( \fint_{B(z,(\lambda_i+1)r)} \rho ~d\mu \right) + o_i(r), \label{e:funct-PI-2}
  \end{align}
  where $o_i$ and $\zeta_i$ have the required properties.

  We have by Jensen's inequality and the definition of *-upper gradient that
  \begin{align*}
    \fint_{B(z,r)} | u - u_{B(z,r)} | ~d\mu &\leq \fint_{B(z,r)} \fint_{B(z,r)} |u(x) - u(y)| ~d\mu(x) ~d\mu(y) \\
    &\leq \fint_{B(z,r)} \fint_{B(z,r)} \inf_{\gamma \in \Gamma(x,y)} \int_\gamma^* \rho ~d\mu(x) ~d\mu(y) \\
    &\leq \fint_{B(z,r)} \fint_{B(z,r)} \tilde{F}_\rho(x,y;\lambda_i) ~d\mu(x) ~d\mu(y).
  \end{align*}
  This, along with \eqref{e:funct-PI-2} finishes the proof.  In the statement of \eqref{e:true-PI}, we have relabeled $\lambda_i + 1$ as $\lambda_i$.
\end{proof}

This behaviour is exactly our notion of an asymptotic non-homogeneous Poincar\'e inequality.

\begin{definition}\label{d:ANPI}
Let $(X,d,\mu)$ be a metric measure space for which porous sets are null.  We then say that $X$ supports an asymptotic non-homogeneous Poincar\'e inequality (or that $X$ is an asymptotic NPI space for short) if it satisfies the conclusion of Theorem \ref{t:true-PI}.
\end{definition}

Recall that, it follows from \cite{porosityandmeasures}, that metric measure spaces in which porous sets are null are automatically pointwise doubling.

\begin{remark}
  It may seem strange that, in the definition of an asymptotic NPI, we must restrict to 1-Lipschitz functions.  This is a side effect of the fact that, to do analysis on Lipschitz functions on possibly disconnected spaces, we need some bound on how the function can fluctuate across gaps.  By setting the Lipschitz constant to be 1, we know then that the function can fluctuate at a rate of at most 1 across any gaps.  These upper bound estimates are crucially used in the definition of asymptotic NPI and so 1 is ``baked'' into some of the terms in the inequality (specifically, the notion of *-upper gradient).  One could of course choose any arbitrarily Lipschitz constant $L \geq 1$ to bound the fluctuation, but this would then change the corresponding terms and so the inequality would need to be rewritten to account for this.
\end{remark}

We now show that if $X$ is an asymptotic NPI space, then $X$ is a RNP-Lipschitz differentiability space.  First, we will need a lemma.

\begin{lemma}\label{lem:pointwise-keith}
  Let $(X,d,\mu)$ be a metric measure space in which all porous sets have measure zero.
  Then for every $x\in X$ there exists a $C_x\geq 1$ such that, for any Lipschitz $f\colon X \to \R$,
  \begin{equation*}
    \Lip(f,x) \leq C_x \lim_{r \to 0} \frac{1}{r} \fint_{B(x,r)} |f - f_{B(x,r)}| ~d\mu
  \end{equation*}
  for $\mu$-a.e.\ $x\in X$.
\end{lemma}

\begin{proof}
  As mentioned above, $(X,d,\mu)$ is pointwise doubling and so, by applying Lemma \ref{l:unif-pointwise-doubling-decomp}, we obtain a countable number of Borel sets $A_i\subset X$ with $\mu(X\setminus \cup_i A_i)=0$ such that $(X,d,\mu)$ is uniformly pointwise doubling at each $x\in A_i$ for each $i\in \N$.  Fix an $i\in \N$ and let $Y=A_i$.

  Proposition 4.3.3 of \cite{keith-lip-lip} says that for any doubling metric measure space, there exists a $C' > 0$ depending on the doubling constant so that for any Lipschitz $f : X \to \R$, we have that
  \begin{equation*}
    \Lip(f,x) \leq C' \lim_{r \to 0} \frac{1}{r} \fint_{B(x,r)} |f - f_{B(x,r)}| ~d\mu
  \end{equation*}
  for almost every $x$.  The same proof shows, in our case, that
  \begin{equation*}
    \Lip(f|_Y,x) \leq C' \lim_{r \to 0} \frac{1}{r} \fint_{B(x,r)} |f - f_{B(x,r)}| ~d\mu.
  \end{equation*}  
  One has to simply restrict any $y,z$ appearing in the proof to the set $Y$ and use Lemma \ref{l:doubling-balls} in place of the doubling condition (however, the integrals in the proof do \emph{not} need to be restricted to $Y$).
  Since all porous subsets of $X$ have measure zero, Lemma \ref{l:porosity-subset} gives
  \begin{equation*}
    \Lip(f,x) = \Lip(f|_Y,x) \leq C' \lim_{r \to 0} \frac{1}{r} \fint_{B(x,r)} |f - f_{B(x,r)}| ~d\mu,
  \end{equation*}
  for $\mu$-a.e.\ $x\in Y$.
  Since the choice of $i\in \N$ was arbitrary, this completes the proof.
\end{proof}

We first show that any asymptotic NPI space is a regular Lipschitz differentiability space.  The proof will resemble the proof of Theorem 10.2 of \cite{bate}.

\begin{lemma} \label{l:PI-LDS}
  Let $(X,d,\mu)$ be an asymptotic NPI space.  Then $(X,d,\mu)$ is a Lipschitz differentiability space.
\end{lemma}

\begin{proof}
  We may suppose without loss of generality that each $\zeta_i$ in the definition of an asymptotic NPI space is continuous and injective.  For the entire proof we fix an $i\in\N$ and a $\zeta_i$ and $U_i \subset X$ appearing in this definition.
  
  Suppose $f \colon X \to \R$ is 1-Lipschitz and let $\rho : X \to [0,1]$ be its *-upper gradient.  
  By applying \eqref{e:true-PI}, Lemma \ref{lem:pointwise-keith}, letting $r \to 0$, and applying Lebesgue's differentiation theorem, we get
  \begin{align}
    \Lip(f,x) \leq C_x \zeta_i(\rho(x)) \label{e:Lip-lip}
  \end{align}
  for almost every $x \in U_i$, for some $C_x\geq 1$.

  Now let $S \subset U_i \cap \{x : \Lip(f,x) > 0\}$ be any set that satisfies $\Hd^1(\gamma \cap S) = 0$ whenever $\gamma \in \Gamma(X)$ has positive $f$-speed.  We will show $\mu(S)$ has measure 0.  By Theorem 8.10 of \cite{bate}, this proves the lemma (see Definition 8.5 of \cite{bate} for the definition of $\tilde{B}$ sets, which are precisely what $S$ are).
  It obviously suffices to show that $S \cap \{x : \Lip(f,x) > \epsilon\}$ has measure 0 where $\epsilon > 0$ is arbitrary.
  
  Let $\delta > 0$ be so that
  \begin{align}
    \delta < \frac{\zeta_i^{-1}(t/C_x)}{t} \label{e:delta-defn}
  \end{align}
  whenever $t \in (\epsilon,1]$.  By the definition of $S$, for any $\gamma \in \Gamma(X)$ we have that
  \begin{align*}
    |(f \circ \gamma)'(t)| \leq \delta \Lip(f,\gamma(t)) \Lip(\gamma,t)
  \end{align*}
  for almost every $t \in \gamma^{-1}(S)$.  Thus, if we define
  \begin{align*}
    \rho(x) = \begin{cases}
      \delta \Lip(f,x) & x \in S \\
      \Lip(f,x) & x \notin S,
    \end{cases}
  \end{align*}
  then $\rho$ is an *-upper gradient of $f$.  However, for almost every $x \in S$, we have that
  \begin{align*}
    \delta \Lip(f,x) = \rho(x) \overset{\eqref{e:Lip-lip}}{\geq} \zeta_i^{-1}\left(\frac{ \Lip(f,x)}{C_x} \right).
  \end{align*}
  This, along with the fact that $\Lip(f,x) > \epsilon$, contradicts \eqref{e:delta-defn}.  Thus, $\mu(S) = 0$.
\end{proof}

\begin{remark} \label{r:lip-lip}
  In \cite{keith-lip-lip}, Keith defined the ``Lip-lip inequality'' on a metric measure space $(X,d,\mu)$ if there exists a $C\geq 1$ such that, for every Lipschitz $f\colon X\to \mathbb R$,
  \[\Lip(f,x) \leq C \lip(f,x)\]
  for $\mu$-a.e. $x\in X$ (recall Section \ref{s:preliminaries} for the definition of $\lip$).  Keith proved that any doubling metric measure space with a Lip-lip inequality is a LDS.  He also showed that any PI space satisfies this condition (thus extending the work of Cheeger).

  More recently, the converse to Keith's result was given in \cite{bate} where the constant $C$ is allowed to vary with the point $x$.  This was then extended by Schioppa \cite{schioppa-derivations} to the ``Lip-lip equality'' (where $C=1$ everywhere).

  It is easy to see that $\rho(x)=\lip(f,x)$ is an *-upper gradient of a Lipschitz function $f$.  Therefore, the above proof shows that a pointwise doubling measure with the non-homogeneous Lip-lip condition as given in \eqref{e:Lip-lip} implies Lipschitz differentiability.  Hence, the non-homogeneous Lip-lip inequality self improves to a Lip-lip equality.
\end{remark}

We can now go from non-homogeneous Poincar\'e inequality back to RNP-LDS.

\begin{theorem} \label{t:PI-RNP-LDS}
  Let $(X,d,\mu)$ be an asymptotic NPI space.  Then $X$ is a RNP-LDS.
\end{theorem}

\begin{proof}
  The proof will resemble the proof of Theorem 1.5 of \cite{cheeger-kleiner-rnp}, however we already have a candidate for a the derivative of a RNP valued Lipschitz function: the gradient with respect to a collection of Alberti representations.

  Let $V$ be a RNP Banach space and $f : X \to V$ be 1-Lipschitz.  As $X$ is separable, $f(X)$ is separable and so we can and will suppose $V$ is separable.  Thus, there exists a countable set $D \subset B_{V^*}$ so that
  \begin{align}
    \|x\| = \sup_{u \in D} |u(x)|, \qquad \forall x \in V. \label{e:norming}
  \end{align}

  By Lemma \ref{l:PI-LDS}, we know that $X$ is a Lipschitz differentiability space and so $X$ admits a countable Borel decomposition into charts $(U_i,\phi_i : X \to \R^{n_i})$ such that $\R$-valued Lipschitz functions on $X$ are differentiable almost everywhere with respect to the charts.

  Fix a chart $(U,\phi : X \to \R^n)$.  Then there exists a full measure set $A \subseteq U$ so that for each $x \in A$ there exist $\gamma_1,...,\gamma_n \in \Gamma(X)$ such that $\gamma_i^{-1}(x)$ is a density point of $\dom \gamma_i$ (which we will call $t$), $u_i(x) = (\phi \circ \gamma_i)'(t)$ are linearly independent in $\R^n$, and $v_i(x) = (f \circ \gamma_i)'(t) \in V$ exist.  For each $x \in A$, we can then then define $\nabla f(x) \in L(\R^n,V)$ so that $\nabla f(x)(u_i(x)) = v_i(x)$.  Moreover, by Proposition 2.9 of \cite{bate}, the assignment $\nabla f\colon x \mapsto \nabla f(x)$ may be chosen in a measurable way.

  By Lusin's theorem, we have that for any $\delta > 0$, there exists a compact set $S \subseteq A$ so that $\mu(S) \geq \mu(U) - \delta$ and $\nabla f$ is continuous on $S$.  For $u \in V^*$, we have that $u(v_i(x)) = u((f \circ \gamma_i)'(t)) = (f_u \circ \gamma_i)'(t)$ by the chain rule.  Thus, we have that $u \circ \nabla f(y)(u_i(x)) = (f_u \circ \gamma_i)'(t)$ and so by Theorem \ref{t:LDS-UAR}, we have that
  \begin{align}
    u \circ \nabla f(x) = Df_u(x). \label{e:composition}
  \end{align}
for $\mu$-a.e. $x\in U$.

  Let $x \in S$ be a density point and $\epsilon > 0$.  By continuity of $\nabla f$ on $S$, there exists some $r > 0$ so that
  \begin{align}
    \|\nabla f(y) - \nabla f(x)\| < \epsilon, \qquad \forall y \in B(x,r) \cap S. \label{e:cont-1}
  \end{align}
  In particular,
  \begin{align}
    |Df_u(y) - Df_u(x)| \overset{\eqref{e:composition}}{=} |u \circ \nabla f(y) - u \circ \nabla f(x)| \overset{\eqref{e:cont-1}}{<} \epsilon, \qquad \forall y \in B(x,r) \cap S, u \in D. \label{e:cont-2}
  \end{align}
  Define the function
  \begin{align*}
    \ell : X &\to V \\
    z &\mapsto f(x) + \nabla f(x) (\phi(z) - \phi(x)),
  \end{align*}
  so that, if $\ell_u = u \circ \ell$, we get that $D\ell_u(y) = u \circ \nabla f(x) = Df_u(x)$.  Note that $\ell$ is 1-Lipschitz.  We have
  \begin{align*}
    |Df_u(y) - D\ell_u(y)| = |Df_u(y) - Df_u(x)| \overset{\eqref{e:cont-2}}{<} \epsilon, \qquad \forall y \in B(x,r) \cap S, u \in D.
  \end{align*}
  In particular, we see from the definition of the derivative of $f_u-\ell_u$ at $y$ that
  \[\Lip(f_u-\ell_u,y) < \epsilon \Lip(\phi,y)\leq\epsilon\Lip \varphi \qquad \forall y \in B(x,r) \cap S, u \in D.\]

  Since the pointwise Lipschitz constant of a 1-Lipschitz function is always an *-upper gradient, we may apply \eqref{e:true-PI} to the 1-Lipschitz function $\frac{1}{2}(f_u - \ell_u)$ and use the fact that $x \in S$ is a density point to get that
  \begin{align*}
    \sup_{y \in B(x,r)} \frac{|u \circ (f(y) - \ell(y))|}{r} = \sup_{y \in B(x,r)} \frac{|f_u(y) - \ell_u(y)|}{r}
  \end{align*}
  converges to 0 as $r \to 0$ uniformly in $u \in D$.  Using \eqref{e:norming} we get that $\nabla f(x)$ is the derivative of $f$ at $x$ with respect to the chart map $\phi$.

  Thus, we get that $f$ is differentiable on a full measure set of points in $S$.  Letting $\delta \to 0$ shows the differentiability of $f$ almost everywhere in $U$.
\end{proof}



By combining Theorems \ref{t:porous-null}, \ref{t:true-PI} and \ref{t:PI-RNP-LDS} we have our first characterisation of RNP-LDS.

\begin{theorem}\label{t:poincarechar}
A metric measure space $(X,d,\mu)$ is a RNP Lipschitz differentiability space if and only if it satisfies an asymptotic non-homogeneous Poincar\'e inequality.
\end{theorem}

\section{Construction of a non-differentiable RNP valued Lipschitz function}\label{s:nondiff}

The goal of this section is to prove Proposition \ref{p:nondiff-construct}.  We show that, if there is a collection of $\R$-valued Lipschitz functions that satisfy the hypotheses of Proposition \ref{p:nondiff-construct}, then we can glue them together in a way to get a \emph{single} RNP-valued function that satisfies similar properties for many points, at a fixed scale.

\begin{lemma} \label{l:scale-nondiff}
  Let $(X,d,\mu)$ be a metric measure space and $0<\delta,\epsilon <1$.  Suppose that $A\subset X$ such that, for almost every $x \in A$, there exists a 1-Lipschitz function $f_x^\epsilon : X \to \R$ and a sequence $(y_i) \subset X$ such that $y_i \to x$,
  \begin{align}
    \Lip(f_x^\epsilon,z) \leq \epsilon, \qquad \mu\mbox{--}a.e. ~z \in X, \label{e:f_xi-lip-upper}
  \end{align}
  and
  \begin{align*}
    |f_x^\epsilon(x)-f_x^\epsilon(y_i)| > \delta d(x,y_i).
  \end{align*}
  Then for any $\eta,r>0$ there exist a subset $A' \subseteq A$ such that $\mu(A') \geq \mu(A) - \eta$, a $p \in [1,\infty)$,  and a 3-Lipschitz map $G : X \to \ell_p$ with the following properties:
  \begin{enumerate}
    \item $\Lip(G,z) \leq 3\epsilon$ for $\mu$-a.e. $z \in X$,
    \item $\|G(x)\|_p \leq r$ for every $x \in X$,
    \item For every $x \in A'$, there exists $y\in X$ with $0<d(x,y)<r$ and
    \begin{align*}
       \frac{\|G(x) - G(y)\|_p}{d(x,y)} > \frac{\delta}{8}.
    \end{align*}
  \end{enumerate}
\end{lemma}

\begin{proof}
  We may suppose without loss of generality that $f_x^\epsilon(x) = 0$ for all $x \in A$ for which $f_x^\epsilon$ exists.  By assumption, we can find some $y_x$ such that $d(y_x,x) < r/12$ and $|f_x^\epsilon(y_x) - f_x^\epsilon(x)| > \delta d(y_x,x)$.  Setting $r_x = 2d(x,y_x)$, we get that $\mu\left( A \backslash \bigcup_{x \in A} B(x,r_x) \right) = 0$.  Note that
  \begin{align}
    r_x \leq \frac{r}{6}, \qquad \forall x \in A. \label{e:r_x-bound}
  \end{align}

  Recall that we are supposing $X$ and hence $A\subset X$ have finite measure.  Thus, there exists some $A' \subseteq A$ such that $\mu(A') \geq \mu(A) - \eta$ and some finite subset $\{x_1,...,x_p\} \subseteq A'$ such that
  \begin{align*}
    A' \subseteq B(x_1,r_{x_1}) \cup ... \cup B(x_p,r_{x_p}).
  \end{align*}
  This $p$ will be the $p$ given by the lemma.  Let $y_i = y_x$ and $r_i = r_{x_i}$.

  Define the function
  \begin{align*}
    G : X &\to \ell_p \\
    x &\mapsto (\min(|f_{x_1}^\epsilon(x) - f_{x_1}^\epsilon(x_1)|,r/3),..., \min(|f_{x_p}^\epsilon(x) - f_{x_p}^\epsilon(x_p)|,r/3),\\
    &\qquad \min(|f_{x_1}^\epsilon(x) - f_{x_1}^\epsilon(y_1)|,r/3),..., \min(|f_{x_p}^\epsilon(x) - f_{x_p}^\epsilon(y_p)|,r/3),0,0,...).
  \end{align*}
  As $G$ is the $\ell_p$-sum of $2p$ 1-Lipschitz functions, we get that
  \begin{align*}
    \|G(x) - G(y)\|_p \leq \left( 2p d(x,y)^p \right)^{1/p} \leq 3 d(x,y),
  \end{align*}
  and so $G$ is 3-Lipschitz.  We also have that
  \begin{align*}
    \|G(x)\|_p \leq \left( 2p (r/3)^p \right)^{1/p} \leq r.
  \end{align*}
  Note that $G$ maps to a subspace supported on a finite set of coordinates of $\ell_p$.  Thus, we have by \eqref{e:f_xi-lip-upper} that
  \begin{align*}
    \Lip(G,z) \leq \left( 2p \epsilon^p \right)^{1/p} \leq 3\epsilon, \qquad \mu\text{-a.e.} ~z \in X.
  \end{align*}

  To prove the last property,  let $x\in A'$ be contained in $B(x_i,r_i)$ and note that, as $(x,y) \mapsto |f_{x_i}^\epsilon(x) - f_{x_i}^\epsilon(y)|$ satisfies the triangle inequality, we get that
    \begin{align}
      \max( |f_{x_i}^\epsilon(x) - f_{x_i}^\epsilon(x_i)|,  |f_{x_i}^\epsilon(x) - f_{x_i}^\epsilon(y_i)| ) > \frac{\delta}{2} d(x_i,y_i). \label{e:xy-ineq}
    \end{align}
    If we let $y\in \{x_i,y_i\}$ achieve the maximum then $d(x,y)\leq 2r_i<r$ and
    \begin{align*}
      |f_{x_i}^\epsilon(x) - f_{x_i}^\epsilon(y)| > \frac{\delta}{2} d(x_i,y_i) = \frac{\delta}{8} 2r_i \geq \frac{\delta}{8} d(x,y).
    \end{align*}

\end{proof}

The following lemma allows us to conclude that, after excising a negligible portion of $X$, a function with small pointwise Lipschitz constant actually has small Lipschitz behavior at small enough scales.

\begin{lemma} \label{l:lipschitz-decomp}
  Let $f : (X,d_X,\mu) \to (Y,d_Y)$ be a Lipschitz function and suppose $\mu(X) < \infty$ and $\Lip(f,x) < \epsilon$ $\mu$-a.e. $x \in X$ for some $\epsilon > 0$.  Then for any $\delta > 0$, there exists a subset $X' \subseteq X$ such that $\mu(X') \geq \mu(X) - \delta$ and an $R > 0$ so that if $x,y \in X'$ for which $d_X(x,y) < R$, then
  \begin{align*}
    d_Y(f(x),f(y)) \leq \epsilon d_X(x,y).
  \end{align*}
\end{lemma}

\begin{proof}
For $R>0$ define the set $A_R$ to be those $x\in X$ such that
  \begin{align*}
    \sup_{y \in B(x,R)} \frac{d_Y(f(x),f(y))}{d_X(x,y)} \leq \epsilon.
  \end{align*}
  Then each $A_R$ is closed and, by assumption, converge to a set of full measure as $R\to 0$.  Thus, there exists some $R > 0$ so that $\mu(A_R) \geq \mu(X) - \delta$.  We let $X' = A_R$.  It follows easily from definition now that if $x,y \in X'$ such that $d_X(x,y) < R$, then $d_Y(f(x),f(y)) \leq \epsilon d_X(x,y)$.
\end{proof}

We now come to the proof of Proposition \ref{p:nondiff-construct}.  Here, we glue the functions from Lemma \ref{l:scale-nondiff} together to get one RNP-valued function with large pointwise Lipschitz constant, but small partial derivatives.  This contradicts our assumption that there exists a RNP-universal collection of Alberti representations.

\begin{proof}[Proof of Proposition \ref{p:nondiff-construct}]
  By Proposition \ref{p:RNP-universal}, there exists a Borel decomposition $X = \bigcup_i U_i$ so that each $U_i$ has a finite collection of RNP-universal Alberti representations.  It suffices to prove the result for $S\subset U_i$, some $i\in \mathbb N$ and so now we fix such an $i$ and let $X_0=U_i$.



  Assume for contradiction that $\mu(S) > 0$.  We will construct a sequence of sets $S_i \subseteq S$, $X_i \subseteq X_0$, a sequence of numbers $r_i,R_i \in (0,1)$, and a sequence of 3-Lipschitz maps $G_i : X \to \ell_{p_i}$ for $p_i \in [1,\infty)$ with the following properties:
  \begin{enumerate}[(i)]
    \item $S_i \subseteq X_i$,
    \item $\mu(S_i) \geq (1 - 2^{-i-1}) \mu(S)$,
    \item $R_i < r_i/2$, \label{e:R<r}
    \item $r_i = 2^{-i} R_{i-1}$, \label{e:r-R}
    \item $\|G_i(x)\| \leq r_i$ for all $x \in X$, \label{e:absolute-bound}
    \item $\|G_i(x) - G_i(z)\|_{p_i} \leq 3\cdot2^{-i+1} d(x,z)$ if $x,z \in X_i$ and $d(x,z) < R_i$, \label{e:lipschitz-bound}
    \item For every $x \in S_i$, there exists $y \in X$ with $d(x,y)<r_i$ and
    \begin{align*}
      \frac{\|G_i(x) - G_i(y)\|_{p_i}}{d(y,z)} > \frac{\delta}{8}.
    \end{align*} \label{e:bad-balls}
  \end{enumerate}

  To do so, for each $m\in\mathbb N$ we iteratively apply Lemma \ref{l:scale-nondiff} with $A = S$ and $\epsilon = 2^{-m}$ to get $A' = \tilde{S}_m \subseteq S$ such that $\mu(\tilde{S}_m) \geq (1 - 2^{-m-2}) \mu(S)$, $r_m = 2^{-m} R_{m-1}$ (with $r_1 = 1$), and a map $G_m : X \to \ell_{p_m}$ that satisfies the conclusions of the lemma.
  We then apply Lemma \ref{l:lipschitz-decomp} to get a set $X_m \subseteq X$ such that
  \begin{align*}
    \mu(X_m) \geq \mu(X) - 2^{-m-2} \mu(S)
  \end{align*}
  and an $R_m < r_m/2$ so that for $x,z \in X_m$, $\|G_m(x)-G_m(z)\|_{p_m} \leq 3\cdot2^{-m+1} d(x,z)$ when $d(x,z) < R_m$.  Take $S_m = X_m \cap \tilde{S}_m \subseteq S$.  Then $\mu(S_m) \geq \mu(\tilde{S}_m) - 2^{-m-2} \mu(S) \geq (1- 2^{-m-1}) \mu(S)$.  It is evident that all the properties are now satisfied.

  Let $S_\infty = \bigcap_{i \in \N} S_i \subseteq S$ and $X_\infty = \bigcap_{i \in \N} X_i \subseteq X$.  We then have that $S_\infty \subseteq X_\infty$,
  \begin{align}
    \mu(S_\infty) \geq \frac{1}{2} \mu(S) > 0. \label{e:Sinfty-bound}
  \end{align}
  For each $k \in \N$, define the map
  \begin{align*}
    F_k : X_\infty &\to \left( \sum_i \ell_{p_i} \right)_1 \\
    x &\mapsto (0,0,...,0,G_k(x),G_{k+1}(x),...)
  \end{align*}
  where the first $k-1$ coordinates are 0.  Note that as $p_i \in [1,\infty)$ for all $i$, we get that the target space is an $\ell_1$-sum of RNP Banach spaces and so is itself RNP (see {\it e.g.} p. 219 of \cite{diestel-uhl}).

  We first claim that each $F_k$ is Lipschitz.  Let $x,y \in X_\infty$.  Suppose $d(x,y) \in [R_{i+1},R_i)$ for some $i$.  Then for $j > i + 1$, we use Property \eqref{e:absolute-bound} to get that
  \begin{align*}
    \|G_j(x) - G_j(y)\|_{p_j} \leq 2r_j \overset{\eqref{e:r-R}}{=} 2^{-j+1} R_j \leq 2^{-j+1} d(x,y).
  \end{align*}
  On the otherhand, for $j \leq i$, we use Property \eqref{e:lipschitz-bound} to get that
  \begin{align*}
    \|G_j(x) - G_j(y)\|_{p_j} \leq 3\cdot2^{-j+1} d(x,y).
  \end{align*}
  As $G_{i+1}$ is 3-Lipschitz, we get that
  \begin{align*}
    \|F_k(x) - F_k(y)\| = \sum_{j=k}^\infty \|G_j(x) - G_j(y)\| \leq \left( 3 + \sum_{j=k}^\infty 2^{-j+1} \right) d(x,y) \leq 4 d(x,y).
  \end{align*}

  We claim that for every biLipschitz $\gamma\in \Gamma$ and $\mathcal L^1$-a.e. $t \in \dom \gamma$,
  \begin{align}
    \|(F_k \circ \gamma)'(t)\| \leq 2^{-k+4} \Lip(\gamma,t). \label{e:small-partial}
  \end{align}
  Fix $t \in \dom(\gamma)$ such that $(F_k \circ \gamma)'(t)$ and $\{(G_j \circ \gamma)'(t)\}_{j=0}^\infty$ all exist and so $(F_k \circ \gamma)'(t) = \sum_{j=k}^\infty (G_j \circ \gamma)'(t) e_j$.  Thus,
  \begin{align*}
    \|(F_k \circ \gamma)'(t)\| \leq \sum_{j=k}^\infty \|(G_j \circ \gamma)'(t)\|_{p_j} = (*).
  \end{align*}
  As $(G_j \circ \gamma)'(t)$ exists, we must have
  \begin{align*}
    \|(G_j \circ \gamma)'(t)\|_{p_j} = \Lip(G_j \circ \gamma,t) \leq \Lip(G_j,\gamma(t)) \Lip(\gamma,t) \overset{\eqref{e:lipschitz-bound}}{\leq} 3 \cdot 2^{-j+1} \Lip(\gamma,t).
  \end{align*}
  Altogether we get
  \begin{align*}
    (*) \leq \sum_{j=k}^\infty 3 \cdot 2^{-j+1} \Lip(\gamma,t) \leq 2^{-k+4} \Lip(\gamma,t).
  \end{align*}

 Property \eqref{e:bad-balls} gives, for every $x \in S_\infty$, a sequence $X\ni y_n \to x$ with
  \begin{align}\label{e:big-Lip}
     \frac{\|G_i(x) - G_i(y_n)\|_{p_i}}{d(x,y)} > \frac{\delta}{8}
  \end{align}
 By Lemma \ref{l:porosity-subset} we can find such $y_n \in X_\infty$ for almost every $x\in S_\infty$.  By Lemma \ref{l:RNPLDS-subset}, $X_\infty$ is an RNP-LDS and so by Proposition \ref{p:RNP-universal} has a countable decomposition into sets each with a collection of RNP-universal Alberti representations.  As $\mu(S_\infty) > 0$, it has a positive measure subset $U$ for which $\mu|_U$ has a RNP-universal set of Alberti representations with constant $\kappa > 0$.  However, we have that each $F_k$ is Lipschitz and if $k$ is sufficiently large, then \eqref{e:small-partial} and \eqref{e:big-Lip} contradict the $\kappa$ RNP-universality bound of $\mu|_U$.  Thus, our initial assumption that $\mu(S) > 0$ is false.
\end{proof}

\section{Tangents of RNP-differentiability spaces}\label{s:tangents}

In this section, we show that tangents of RNP-Lipschitz differentiability spaces are quasiconvex RNP-Lipschitz differentiability spaces.  We first review tangent spaces.  To do this, we must recall the notion of pointed Gromov-Hausdorff convergence.

We say that a sequence of pointed metric measure spaces $(X_i,x_i)$ converge to $(Y,y)$ in the pointed Gromov-Hausdorff sense if there exist isometric embeddings $\iota_i : (X_i,x_i) \to (Z,z)$ and $\iota : (Y,y) \to (Z,z)$ into another pointed metric space $(Z,z)$ such that for each $R > 0$, we have Hausdorff convergence in $B(z,R)$:
\begin{align*}
  \lim_{i \to \infty} \sup_{y \in B(z,R) \cap \iota(Y)} \dist_Z\left( \{y\}, \iota_i(X_i) \right) &= 0, \\
  \lim_{i \to \infty} \sup_{y \in B(z,R) \cap \iota_i(X_i)} \dist_Z\left( \{y\}, \iota(Y) \right) &= 0.
\end{align*}
We then say that $(Y,y)$ is the pointed Gromov-Hausdorff limit of $(X_i,x_i)$.

For $\lambda > 0$ and metric space $(X,d)$, we let $\frac{1}{\lambda}X$ denote the metric space $(X,d')$ where $d'(x,y) = \frac{1}{\lambda} d(x,y)$.  Let $X$ now also have a measure $\mu$.  Then we say $(Y,\nu,y)$ is a tangent space of $(X,\mu)$ at $x$ if there is a sequence $\lambda_i \to 0$ so that $( \lambda_i^{-1} X,x )$ pointed Gromov-Hausdorff converges to $(Y,\rho,y)$ and for the defining isometric embeddings $\iota_i : (\lambda_i^{-1} X,x) \to (Z,z)$ and $\iota : (Y,y) \to (Z,z)$ we have the weak convergence in $Z$
\begin{align*}
  (\iota_i)_\# \frac{\mu}{\mu(B(x,\lambda_i))} \rightharpoonup \iota_\# \nu.
\end{align*}

Note that the tangent space can depend on the sequence $\lambda_i \to 0$ and so need not be unique.

\begin{remark}[Tangents of pointwise doubling metric spaces]
If $(X,d,\mu)$ is pointwise doubling, then by lemma \ref{l:unif-pointwise-doubling-decomp}, there exists a countable number of Borel sets $A_i\subset X$ with $\mu(X\setminus \cup_i A_i)=0$ such that each $A_i$ is complete and metrically doubling.  Thus we can apply Gromov's tangent theory inside each $A_i$.  Note that we must work inside the $A_i$, in particular because we have no control over the null set that is not contained within any $A_i$.

However, such tangents are unique in the following sense.  Let $\mu(X\setminus \cup_i A_i)=\mu(X\setminus \cup_i B_i)=0$ be two such decompositions that in addition satisfy the uniform doubling estimates ensured by Lemma \ref{l:unif-pointwise-doubling-decomp}.  Then if $x$ is a density point of $A_i\cap B_j$ for some $i,j$ then $\Tan(A_i,x,\mu)=\Tan(B_j,x,\mu)$.  This is a consequence of the fact that tangents at density points of subsets agree with the tangents of the big space, see \cite{ledonne}*{Proposition 3.1}.  Note that this result of Le Donne is written for doubling measures, however the proof only relies on the estimates given by Lemma \ref{l:unif-pointwise-doubling-decomp} and so the result is true in our case.
\end{remark}

Under the assumption that $(X,\mu)$ is complete and pointwise doubling (as is our case), we get that tangent spaces exist for almost everywhere and are complete and doubling.  We let $\Tan(X,\mu,p)$ denote the set of tangent spaces of $(X,\mu)$ at $p \in X$.


First we see that the *-integral is well behaved under taking a limit.  Recall that the metric on $\Gamma$ is Hausdorff distance between graphs.  See \cite{bate}*{Definition 2.1}.

\begin{lemma}\label{l:starlowersemicts}
Let $(X,d)$ be a metric space, $\rho\colon X\to [0,1]$ continuous and  Suppose that $\gamma_n \to \gamma_\infty \in \Gamma$.  Then
\[\int_{\gamma_\infty}^* \rho \leq \liminf_{n\to \infty} \int_{\gamma_n}^* \rho.\]
\end{lemma}

\begin{proof}
Let $\epsilon>0$.  Since $\im\gamma_\infty$ is compact, there exists a $0<\delta_1<1$ such that
\begin{equation}\label{e:rhounifcts}
d(\gamma_\infty(t),z)< \delta \Rightarrow |\rho(\gamma_\infty(t))-\rho(z)|<\epsilon.\end{equation}
Let $\delta_2>0$ such that the open neighbourhood $B(\dom\gamma_\infty,\delta_2)$ satisfies
\begin{equation}\label{e:smallnhood}|B(\dom\gamma_\infty,\delta_2)\setminus\dom\gamma_\infty|<\epsilon,\end{equation}
let $\delta=\min(\delta_1,\delta_2,\epsilon)$ and $N\in\N$ such that $d(\gamma_\infty,\gamma_n)<\delta$ for all $n>N$.  Finally we let $a=\inf B(\dom\gamma_\infty,\delta_2)$ and $b=\sup B(\dom\gamma_\infty,\delta_2)$.  Note that necessarily $b-a \leq \diam \dom\gamma_\infty+2$ since $\delta\leq \delta_1 < 1$.

For each $n\in\N \cup\{\infty\}$ define $\rho_n\colon [a,b]\to [0,1]$ by
\[\rho_n(t) = \begin{cases} \rho(\gamma_n(t)) & t\in\dom\gamma_n\\
1 & \text{otherwise}\end{cases}.\]
Then
\[\left|\int_{\gamma_n}^* \rho - \int_a^b \rho_n\right| \leq |(a,b)\setminus \dom\gamma_n| \leq \epsilon\]
for each $n\in\N\cup\{\infty\}$.  For each $n\in\N$, there are four cases to consider for $t\in[a,b]$:
\begin{enumerate}
  \item If $t\in\dom\gamma_\infty \cap \dom\gamma_n$ then $|\rho_n(t)-\rho_\infty(t)|<\epsilon$ by \eqref{e:rhounifcts}.
  \item If $t\not\in \dom\gamma_\infty \cup \dom\gamma_n$ then $\rho_n(t)=\rho_\infty(t)$.
  \item If $t\in\dom\gamma_\infty \setminus \dom\gamma_n$ then
  \[\rho_n(t) = \rho(\gamma_n(t)) \leq 1 = \rho_\infty(t).\]
  \item If $t\in\dom\gamma_n\setminus \dom\gamma_\infty$ then $t\in B(\dom\gamma_\infty,\delta_2)\setminus \dom\gamma_\infty$, which has measure less than $\epsilon$ by \eqref{e:smallnhood}.
\end{enumerate}

Therefore
\begin{align*}\int_a^b \rho_n -\rho_\infty &\geq -\int_{(1)} \epsilon + \int_{(2)} 0 + \int_{(3)}0 - \int_{(4)}1\\
&\geq -\epsilon(\diam\dom\gamma_\infty+2) -\epsilon\end{align*}
and so
\[\int_{\gamma_n}^* \rho -\int_\gamma^* \rho \geq -\epsilon(\diam\dom\gamma_\infty + 4).\]
\end{proof}

We have the following proposition.

\begin{proposition} \label{l:limit-PI-frag}
  Suppose $X$ is a RNP-LDS.  Then for almost every $x \in X$ there exists a $\tilde{\zeta} : [0,\infty) \to \R$ and $\lambda \geq 1$ so that for every $(Y,d,\nu) \in \Tan(X,\mu,x)$, one has for all $w \in Y$, $r > 0$, and $\rho : Y \to [0,1]$ that
  \begin{align}
    \fint_{B(w,r)} \fint_{B(w,r)} \tilde{F}_\rho(u,v;\lambda) ~d\nu(u) ~d\nu(v) \leq \tilde\zeta \left( \fint_{B(w,4\lambda r)} \rho ~d\nu \right). \label{e:limit-tPI}
  \end{align}
\end{proposition}

\begin{proof}

  Let $X=\cup_i U_i$ be the decomposition given by Theorem \ref{t:ae-poincare}, fix some $i\in \N$ and let $K\subset U_i$ be compact with uniform pointwise doubling constants.  Then for any density point $x$ of $K$, $\Tan(X,\mu,x)=\Tan(K,\mu,x)$ and so we may suppose $K=X$ and fix $\lambda=\lambda_i$, $\zeta=\zeta_i$, and $o = o_i$.  Without loss of generality, we may assume $\zeta$ is continuous.  By \cite{cheeger}*{Lemma 5.18}, we may suppose that $\rho$ is continuous.  The proof of this lemma holds in our situation as written provided one interprets convergence and compactness to be in $\Gamma$ (which is compact because $X$ is).

  Let $(Y,d,\nu,y)\in\Tan(X,\mu,x)$.  Then there exists a pointed metric space $(Z,z)$ and isometric embeddings $\iota_i : (s_i^{-1}X_i,x)\to (Z,z)$ and $\iota: (Y,y) \to (Z,z)$.  From now on, we will identify $Y$ and $X_i := s_i^{-1}X$ with their isometric images in $Z$ and we also view $\mu_i$ and $\nu$ as measure on $Z$ via pushfoward.  We will continuously extend $\rho$ on $Y$ to all of $Z$ so it is also defined on each $X_i$.

  Let $u,v \in Z$ and $A \subseteq Z$ be a closed subset.  We define
  \begin{align*}
    F_\rho^A(u,v) = \inf_{\gamma \in \Gamma(u,v)} \int_{\gamma \cap A}^* \rho
  \end{align*}
  where $\gamma \cap A$ is the restricted curve $\gamma|_{\gamma^{-1}(A)}$ except we do not throw away the endpoints (in case $u,v \notin A$).  It follows easily from the definition of the *-integral that $F_\rho$ is 1-Lipschitz on $Z \times Z$.  As $A$ is closed, $\gamma \cap A$ still has a compact domain.  It follows from definitions that for $u,v \in A$ we have
  \begin{align*}
    F_\rho^A(u,v) = \inf_{\gamma \in \Gamma_A(u,v)} \int_\gamma^* \rho = \tilde{F}_\rho(u,v)
  \end{align*}
  where the $\tilde{F}_\rho$ is the usual definition of $\tilde{F}_\rho$ in  the metric space $A$ with respect to the restricted function $\rho|_A$.  Thus, we are trying to prove that
  \begin{align*}
    \fint_{B(w,r)} \fint_{B(w,r)} F^Y_\rho(u,v;\lambda) ~d\nu(u) ~d\nu(v) \leq \tilde{\zeta}\left( \fint_{B(w,(\lambda+1)r)} \rho ~d\nu \right).
  \end{align*}

  Let $u,v \in Y \subset X$.  For each $i\in\N$ let $\gamma_i\in\Gamma(u_i,v_i;\lambda)$ with
  \[\int_{\gamma_i \cap X_i}^* \rho \leq F_\rho^Y(u_i,v_i;\lambda) + \frac{s_i d(u_i,v_i)}{i}\]
  and, by taking a subsequence if necessary, let $\gamma\in \Gamma(u,v;\lambda)$ such that $\gamma_i \cap X_i \to \gamma$ in $\Gamma(Z)$.  Note that $\im (\gamma_i \cap X_i) \subset X_i \cup Y$ are compact and as $X_i$ converges to $Y$ in Hausdorff distance, we get that $\im \gamma \subset Y$.  Then by Lemma \ref{l:starlowersemicts},
  \begin{align}
    F_\rho^Y (u,v;\lambda) \leq \int_\gamma^* \rho \leq \liminf_{i\to\infty} \int_{\gamma_i \cap X_i}^* \rho=\liminf_{i\to\infty} F_\rho^{X_i}(u,v;\lambda), \qquad \forall u,v \in Y. \label{e:F_rho-lsc}
  \end{align}

  Let $\epsilon > 0$ and $w_i \in X_i \subset Z$ be so that $w_i \to w \in Y \subset Z$.  Then by the continuity of $\rho$, $\zeta$ and the weak convergence of $\mu_i\to\nu$, we have that there exists some $N_1 \in \N$ so that
  \begin{align}
    \zeta \left( \fint_{B(w_i,2(\lambda+1)r)} \rho ~d\mu_i \right) \leq \zeta \left( \frac{\mu(B(w,4 \lambda r))}{\mu(B(w,2(\lambda +1)r))} \fint_{B(w,4\lambda r)} \rho ~d\nu \right) + \frac{\epsilon}{4}, \qquad \forall i \geq N_1. \label{e:rho-weak-converge}
  \end{align}
  We inflated the ball on the right hand side so that $B(w_i,(\lambda+1)r) \subseteq B(w,2 \lambda r)$ to use weak convergence.  This happens for sufficiently large $i$.  We also have that
  \begin{align*}
    \fint_{B(w_i,2r)} \fint_{B(w_i,2r)} F_\rho^{X_i}(u,v;\lambda) ~d\mu_i(u) ~d\mu_i(v) \leq \zeta \left(\fint_{B(w_i,2(\lambda+1)r)} \rho ~d\mu_i \right) + o(s_i r),
  \end{align*}
  since the non-homogeneous Poincar\'e inequality holds in each $X_i$.  Thus, as $o(s_i r) \to 0$, we can choose some $N_2 \in \N$ so that for all $i \geq N_2$, we get
  \begin{align}
    \fint_{B(w_i,2r)} \fint_{B(w_i,2r)} F_\rho^{X_i}(u,v;\lambda) ~d\mu_i(u) ~d\mu_i(v) \leq \zeta \left(\fint_{B(w_i,2(\lambda+1)r)} \rho ~d\mu_i \right) + \frac{\epsilon}{4}. \qquad \forall i \geq N_2. \label{e:PI-converge}
  \end{align}

  Note that $F_\rho^{X_i}(u,v;\lambda),F_\rho^Y(u,v;\lambda) \leq d(u,v)$ always as we can just choose a curve fragment that are the two endpoints $u$ and $v$.  Thus, by weak convergence of $\mu_i \to \nu$, we can choose some $N_3 \in \N$ so that
  \begin{multline}
    \left(\frac{\mu(B(w_i,r))}{\mu(B(w_i,r))} \right)^2 \fint_{B(w,r)} \fint_{B(w,r)} F_\rho^{X_i}(u,v;\lambda) ~d\nu(u) ~d\nu(v) \\
    \leq \fint_{B(w_i,2r)} \fint_{B(w_i,2r)} F_\rho^{X_i}(u,v;\lambda) ~d\mu_i(u) ~d\mu_i(v) + \frac{\epsilon}{4}, \quad \forall i \geq N_3. \label{e:F_rho-weak-converge}
  \end{multline}
  Again, we used that $B(w,r) \subseteq B(w_i,2r)$ to apply weak convergence.  Finally, by \eqref{e:F_rho-lsc} and Fatou's lemma, we can find a $N_4 \in \N$ so that
  \begin{multline}
    \fint_{B(w,r)} \fint_{B(w,r)} F_\rho^Y(u,v;\lambda) ~d\nu(u) ~d\nu(v) \\
    \leq \fint_{B(w,r)} \fint_{B(w,r)} F_\rho^{X_i}(u,v;\lambda) ~d\nu(u) ~d\nu(v) + \frac{\epsilon}{4}, \quad \forall i \geq N_4. \label{e:F_rho-fatou}
  \end{multline}
  We prove our theorem by combining \eqref{e:rho-weak-converge} \eqref{e:PI-converge} \eqref{e:F_rho-weak-converge} \eqref{e:F_rho-fatou} and letting $\epsilon \to 0$.  Note that we will have $\tilde{\zeta}(t) = \left( \frac{\mu(B(w_i,2r))}{\mu(B(w_i,r))}\right)^2 \zeta\left( \frac{\mu(B(w,4\lambda r))}{\mu(B(w,2(\lambda+1)r))} t \right)$, but all the volume ratios are bounded from above by a constant depending only on $\lambda \geq 1$, the doubling constant of $Y$, and the uniform pointwise doubling constant of $X$.
\end{proof}

Lemma \ref{l:limit-PI-frag} allows us to show that tangent spaces are quasiconvex.

\begin{corollary} \label{c:tangent-qc}
  Suppose $X$ is a RNP-LDS.  Then for almost every $x \in X$, $(Y,d,\nu) \in \Tan(X,\mu,x)$ is quasiconvex.
\end{corollary}

\begin{proof}
  Let $u,v \in Y$.  Applying \eqref{e:limit-tPI} to $\rho \equiv 0$ with $B = B(u,2d(u,v))$ gives that
  $$\int_B \int_B \tilde{F}_\rho(x,y;\lambda) ~d\mu ~d\mu = 0.$$
  As $\tilde{F}_\rho \geq 0$ and is Lipschitz, we get then that we must have $\tilde{F}_\rho \equiv 0$ on $B \times B$ and so
  \begin{align*}
    \inf_{\gamma \in \Gamma(u,v;\lambda)} \int_\gamma^* 0 = 0.
  \end{align*}

  As curve fragments in $\Gamma(x,y;\lambda)$ have compact images in $Y$ and $Y$ is proper, we get that there exists some $\gamma \in \Gamma(u,v;\lambda)$ so that
  \begin{align*}
    \int_\gamma^* 0 = 0.
  \end{align*}
  Then by definition, we have that $\ms \gamma = \len \gamma$.  As the domain of $\gamma$ is compact in $\R$, it must be an interval and so $\gamma$ is a rectifiable curve so that $\len \gamma \leq \lambda d(x,y)$.  This proves quasiconvexity of $Y$.

\end{proof}

We now show that we can complete the *-integral into a true integral over full curves in a quasiconvex space.

\begin{lemma} \label{l:qc-frag}
  If $X$ is $C$-quasiconvex and $\rho : X \to [0,1]$, we have that
  \begin{align*}
   \inf_\gamma \int_\gamma \rho \leq C \inf_{\gamma \in \Gamma(x,y)} \int_\gamma^* \rho
  \end{align*}
  where the infimum on the left hand side is taken over all 1-Lipschitz curves from $x$ to $y$.
\end{lemma}

\begin{proof}
  Let $\epsilon > 0$ and take $\tilde{\gamma} : K \to X$ in $\Gamma(x,y)$ so that
  \begin{align*}
    \int_{\tilde{\gamma}}^* \rho \leq \inf_{\gamma \in \Gamma(x,y)} \int_\gamma^* \rho + \epsilon.
  \end{align*}
  Let $\{(a_i,b_i)\}_{i=1}^\infty$ be the open intervals making up $[\inf K, \sup K] \backslash K$ so that $b_i \leq a_{i+1}$.  For each $(a_i,b_i)$, we can find a rectifiable curve $\gamma_i$ from $\gamma(a_i)$ to $\gamma(b_i)$ of length no more than $C |a_i-b_i|$.  We then define a new rectifiable curve $\gamma : I \to X$ which attaches these $\gamma_i$ to the path of $\tilde{\gamma}$ at the appropriate gaps and reparameterizing so that $\gamma$ is 1-Lipschitz.  Taking the integral, we get
  \begin{multline*}
    \int_\gamma \rho = \int_{\tilde{\gamma}} \rho + \sum_{i=1}^\infty \int_{\gamma_i} \rho \leq \int_{\tilde{\gamma}} \rho + \sum_{i=1}^\infty C |a_i - b_i| = \int_{\tilde{\gamma}} \rho + C (\len \tilde{\gamma} - \ms \tilde{\gamma}) \\
    \leq C \int_{\tilde{\gamma}}^* \rho \leq C\inf_{\gamma \in \Gamma(x,y)} \int_\gamma^* \rho + C\epsilon
  \end{multline*}
  In the first inequality, we used the fact that $\rho \leq 1$.  Taking $\epsilon \to 0$ completes the proof.
\end{proof}

Given a function $\rho : X \to [0,1]$, we can define
\begin{align*}
  F_\rho(x,y) = \inf_\gamma \int_\gamma \rho
\end{align*}
where the infimum is taken over all 1-Lipschitz $\gamma$ from $x$ to $y$.  Note that if $X$ is quasiconvex, then this quantity is finite for all $x$ and $y$.

Lemma \ref{l:limit-PI-frag}, Corollary \ref{c:tangent-qc}, Lemma \ref{l:qc-frag}, and Theorem \ref{t:PI-RNP-LDS} then gives us the following theorem.  Namely that tangents to RNP-LDS are non-homogeneous PI (NPI) spaces.

\begin{theorem}\label{t:tangentsareAPI}
  Suppose $X$ is a RNP-LDS.  Then for almost every $x \in X$ there exists a $\zeta : [0,\infty) \to \R$ and $\lambda \geq 1$ so that for every $(Y,d,\nu) \in \Tan(X,\mu,x)$, $Y$ is a quasiconvex doubling RNP-LDS and has for all $z \in Y$, $r > 0$, and $\rho : Y \to [0,1]$ that
  \begin{align*}
    \fint_{B(z,r)} \fint_{B(z,r)} F_\rho(x,y) ~d\nu(x) ~d\nu(y) \leq \zeta \left( \fint_{B(z,\lambda r)} \rho ~d\nu \right).
  \end{align*}
\end{theorem}

\begin{remark}
Although the curves that form this Poincar\'e inequality are not fragmented, it is essential that the inequality only hold for $\rho$ that are bounded in supremum by some fixed number (for example, $1$).  If it were to hold for any bounded $\rho$ then a simple scaling argument would show that we in fact have a 1-Poincar\'e inequality in the tangent.  This is impossible by the result of Schioppa \cite{schioppa-pi} that constructs, for any $p\geq 1$, a space whose tangents (and the original space) satsify a $p'$-Poincar\'e inequality for every $p'>p$ but not for $p$.
\end{remark}

\begin{remark}
As a corollary we see that for almost every $x$ in a RNP-LDS, every element of $\Tan(X,\mu,x)$ is also a RNP-LDS.  Shortly after the first preprint of this paper appeared, Schioppa \cite{schioppa-tangents} proved that for almost every $x$ in a LDS, every element of $\Tan(X,\mu,x)$ is a LDS.
\end{remark}

\section{Restricted Alberti representations}\label{s:restrictedcurves}

We conclude this paper with a second characterisation of RNP-LDS.  Like our first characterisation using the Poincar\'e inequality, it is based upon connecting points in the space by curve fragments.  However, the significant difference is that it only considers those fragments that are ``significant'' from the point of view of the Alberti representations in the space.

These significant fragments will be the ones that produce the gradient of partial derivatives of any Lipschitz function at almost every point.  By using the fundamental theorem of calculus along such a curve fragment, we then see how the partial derivatives in fact form a total derivative; essentially by the same reason as why one may use the fundamental theorem of calculus along axis parallel lines in $\mathbb R^n$ to prove Rademacher's theorem.

To determine if a collection $\Gamma'$ of curve fragments are ``significant'' from the point of view of an Alberti representation $(\bP,\{\mu_\gamma\})$, we will need to restrict the Alberti representation to $\Gamma'$ and compare the resulting measure to $\mu$.

We will also wish to restrict to a subset of the domain of a curve fragment $\gamma$.  Although this restricted fragment is a fragment in its own right, call it $\tilde\gamma$, its relationship with $\bP$ and $\mu_\gamma$ may be completely different to that of $\gamma$.

To illustrate this problem, consider the Alberti representation of $[0,1]$ given by a single curve $\gamma(t)=t$ for $t\in [0,1]$, $\mu_\gamma = \mathcal L^1\llcorner [0,1]$ and $\bP = \delta_\gamma$.  If we wish to restrict the domain of $\gamma$ to $[0,1/2]$, then simply considering the curve $\tilde\gamma(t)=t$ for $t\in [0,1/2]$ is meaningless: $\mu_{\tilde\gamma}=0$ and $\bP(\tilde\gamma)=0$.  Instead, we must consider a subset of the domain of $\gamma$.

Since the original idea of restricting to $\Gamma'$ can also be formulated by restricting domains, we arrive to the following definition.

\begin{definition}
  Let $\mathcal K$ be the set of non-empty compact subsets of $\R$ equipped with the Hausdorff metric.  A \emph{restriction of curves} is a Borel function $\sigma \colon \Gamma \to \mathcal K\cup \{\emptyset\}$ (that is, $\sigma^{-1}(\mathcal K)$ is a Borel set and $\sigma|_{\mathcal K}$ is a Borel function) such that $\sigma(\gamma) \subset \dom \gamma$ for each $\gamma\in \Gamma$.
\end{definition}

We consider the measure induced by restricting Alberti representations to a restriction of curves.

\begin{definition}
Let $(X,d,\mu)$ have an Alberti representation $\cA = (\bP, \mu_\gamma)$ and $\sigma$ be a restriction of curves.  We define $\cA$ \emph{restricted to} $\sigma$ to be the measure
\[\cA \llcorner \sigma = \int_{\Gamma} \mu_\gamma \llcorner \gamma(\sigma(\gamma)) \text{d}\bP(\gamma).\]
\end{definition}

\begin{lemma}
If $\cA = (\bP, \mu_\gamma)$ is an Alberti representation of $(X,d,\mu)$ and $\sigma$ a restriction of curves then $\cA\llcorner \sigma$ is a Borel measure on $X$.  Moreover, $(\bP, \mu_\gamma\llcorner \gamma(\sigma(\gamma)))$ is an Alberti representation of $\cA \llcorner \sigma$.
\end{lemma}

\begin{proof}
Let $\sigma$ be a restriction of curves and $C\subset X$ closed.  The function $\gamma\mapsto \gamma(\sigma(\gamma))$ into $X$ is Borel by composition and hence $\gamma\mapsto \gamma(\sigma(\gamma))\cap C$ is Borel too, also by composition.  Therefore $\gamma\mapsto \mu_\gamma(\gamma(\sigma(\gamma))\cap C$ is Borel and so $\cA\llcorner \sigma(C)$ is well defined.  The sets for which $\cA\llcorner \sigma$ is defined form a $\sigma$-algebra and so include the Borel sets.

Once we know that $\cA$ is well defined, it is clear that it is a measure and it has the given Alberti representation.
\end{proof}

\begin{definition} \label{d:abs-cont}
Suppose that $(X,d,\mu)$ has $n$ Alberti representations $\cA_1,\ldots,\cA_n$ and that $\sigma_1,\ldots,\sigma_n$ are restrictions of curves.  We define $\mu^*$ to be the part of $\mu$ that is absolutely continuous with respect to each $\cA_i \llcorner \sigma_i$.  That is, for each $1\leq i \leq n$ write
\[\mu = \mu \llcorner A_i + \mu \llcorner A_{i}^{c},\]
where $\mu \llcorner A_{i} \ll \cA_i \llcorner \sigma_i$ and $\mu \llcorner A_{i}\perp \cA_i \llcorner \sigma_i$ for each $1 \leq i \leq n$.  Let $A^{*}=A_{1}\cap \ldots \cap A_{n}$ and set $\mu^{*}=\mu \llcorner A^{*}$.  Then $\mu^* \leq \cA_i \llcorner \sigma_i$ for each $1\leq i \leq n$.  By applying \cite{bate}*{Lemma 2.3} (with Radon-Nikodym derivative $\leq 1$) we obtain an Alberti representation $\cA_i^* = (\bP_i,\mu_\gamma^*)$ of $\mu^*$ with $\mu_\gamma^* \leq \mu_\gamma \llcorner \gamma(\sigma_i(\gamma))$.
\end{definition}

To give some intuition behind this definition, we consider the following example.  Suppose that $X=[0,1]^{2}$ and that $\mu$ is Lebesgue measure on $X$ with $\mathcal A_{1}$ the Alberti representation (given by Fubini's theorem) supported on the horizontal line segments and $\mathcal A_{2}$ the Alberti representation supported on vertical line segments.
Also let $\sigma_{1}$ be the restriction of curves that restricts each horizontal segment to its left half and $\sigma_{2}$ the restriction of curves that restricts each vertical segment to its bottom half (and extended to the whole of $\Gamma$ arbitrarily).

Then $\mathcal A_{1} \llcorner \sigma_{1}$ is an Alberti representation of Lebesgue measure on the left half of the square and $\mathcal A_{2}\llcorner \sigma_{2}$ is an Alberti representation of Lebesgue measure on the bottom half of the square.  In particular, $\mu^{*}$ is Lebesgue measure restricted to the bottom left corner $[0,1/2]^{2}$.

Thus, $\mu^{*}$ is the part of $\mu$ that is significant for all of the restricted Alberti representations simultaneously.
Of course, except for very particular choices of the $\sigma_i$, $\mu^*$ will usually be trivial.  Necessarily, the curves that we restrict to must have positive $\bP_i$ measure, but they must also ``overlap'' in a set of positive $\mu$ measure.

There is an obvious decomposition result for restrictions of curves.
\begin{lemma}\label{l:decomposerestrictions}
For $n\in \mathbb N$, let $(X,d,\mu)$ be a metric measure space with $n$ Alberti representations  $(\bP_i,\{\mu_{\gamma,i}\})$.  For each $1\leq i \leq n$ let $\sigma_i^j$, $j\in \mathbb N$, be restrictions of curves such that, for $\bP_i$-a.e. $\gamma\in \Gamma$,
\[\mu_{\gamma,i} \left( \dom\gamma\setminus \bigcup_{j\in\N} \sigma_i^j(\gamma) \right) =0.\]
For each $j=(j_1,\ldots,j_n) \in\N^n$ let $\mu^{*,j}$ be the measure induced by $\sigma^{j_1},\ldots,\sigma^{j_n}$.  Then if $S\subset X$ is Borel with $\mu^{*,j}(S)=0$ for every $j\in\N^n$, $\mu(S)=0$.
\end{lemma}

\begin{proof}
We prove the result by induction on $n$ and first prove the result for $n=1$.  If $\mu^{*,j}(S)=0$ for every $j\in\N$ then, for $\bP_1$-a.e. $\gamma\in\Gamma$ we have $\mu_{\gamma,1}(\sigma^j(\gamma)\cap\gamma^{-1}(S))=0$.  Therefore, for $\bP_1$-a.e. $\gamma\in\Gamma$,
\[\mu_{\gamma,1}(\gamma^{-1}(S))\leq\sum_{j\in\N} \mu_{\gamma,1}(\sigma^j(\gamma)\cap\gamma^{-1}(S)) = 0.\]
In particular $\mu(S)=0$.

For the general case, suppose the result is true for $n-1$.  We first fix $k=(j_1,\ldots,j_{n-1})$ and for each $j\in \N$ let $\mu_j^*$ be the measure induced by $\sigma^{j_1},\ldots,\sigma^{j_{n-1}},\sigma^j$.  Then, by the $n=1$ case applied to $\mu^{*,k}$ and its Alberti representation induced by $(\bP_n,\{\mu_{\gamma,n}\})$, if $\mu^*_j(S)=0$ for every $j\in\N$ then $\mu^{*,k}(S) = 0$.  Therefore, if $\mu^{*,k'}(S)=0$ for every $k'\in \N^{n}$, we must have $\mu^{*,k}(S)=0$ for every $k\in\N^{n-1}$.  By the induction hypotheses we then have $\mu(S)=0$, as required.
\end{proof}

For the rest of this section, we fix a metric measure space $(X,d,\mu)$, Alberti representations $\cA_1,\ldots,\cA_n$ of $\mu$ and restrictions of curves $\sigma_1,\ldots,\sigma_n$.  All other definitions in this section will be implicitly given with respect to this choice.

We now consider a modification of $\rho_\epsilon$ that only gives weight to the part of a curve fragment contained within $\sigma(\gamma)$.  However, for reasons similar to those given at the start of the section, to do this we must consider concatenations of those curve fragments that are present in the restriction of curves.  Naturally, we must allow for translations of the domains of the curves that we concatenate together.

\begin{definition}
We define $\tilde\Gamma$ to be the set of all $\gamma \in \Gamma$ for which there exist $N\in \mathbb N$, $\gamma_1,\ldots, \gamma_N \in \Gamma$, $a_1,\ldots,a_N\in \R$, $t_1<\ldots <t_N \in \R$ and $\ell_1,\ldots,\ell_N \in \{1,\ldots,n\}$ such that:
\begin{itemize}
  \item The domain of $\gamma$ is covered by translations of restrictions of the domains of the $\gamma_i$:
  \[\dom\gamma \subset \bigcup_{i=1}^N [t_i,t_{i+1}]\cap ((t_i-a_i)+\sigma_{\ell_i}(\gamma_i)),\]
  \item For every $1\leq i \leq N$, when restricted to each $\dom\gamma \cap[t_i,t_{i+1}]$, $\gamma$ agrees with the translation of $\gamma_i$:
  \[\gamma(t) = \gamma_i(t+a_i-t_i) \quad \forall t \in \dom\gamma \cap[t_i,t_{i+1}].\]
\end{itemize}
For $x,y \in X$ and $\lambda \geq 1$, we define
\[\tilde\Gamma(x,y)= \tilde\Gamma\cap \Gamma(x,y) \text{ and }  \tilde{\Gamma}(x,y;\lambda) = \tilde\Gamma \cap \Gamma(x,y;\lambda).\]
Finally, for $0<\epsilon < 1$ we define
\[\tilde\rho_\epsilon(x,y) = \inf_{\gamma\in\tilde\Gamma(x,y)} \int_\gamma^* \epsilon.\]
As for $\rho_\epsilon$, $\tilde\rho_\epsilon$ is 1-Lipschitz in each argument and defines a metric on $X$.
\end{definition}

Note that $\tilde\rho_\epsilon$ agrees with $\rho_\epsilon$ (and $\tilde\Gamma$ with $\Gamma$) whenever we consider the trivial restrictions of curves $\sigma_i(\gamma) = \dom \gamma$ for each $\gamma\in\Gamma$ and $1\leq i\leq n$.  Note also that we do not impose an upper bound on the number of concatenations permitted to form an element of $\Gamma^*$.  Indeed, we will have no control over the number of concatenations required to connect two arbitrary points $x$ and $y$; we will instead be interested in the quantity $\tilde\rho_\epsilon(x,y)$.

\begin{remark}
Notice that, by the way we have defined $\tilde\Gamma$, we automatically have $\gamma|A \in \tilde\Gamma$ whenever $\gamma \in \tilde\Gamma$ and $A \subset \R$ is Borel.
\end{remark}

The important fact about restricting Alberti representations and Cheeger differentiability is that the universal property with respect to the restriction of curves holds $\mu^*$ almost everywhere.

\begin{lemma}
Let $\cA$ be an Alberti representation of $(X,d,\mu)$ and $\sigma$ a restriction of curves.  If for some Lipschitz $\phi\colon X \to \mathbb R^n$ and cone $C\subset \mathbb R^n$, $\mathcal A$ is in the $\phi$-direction of $C$, then so is $\cA\llcorner \sigma$.  If for $\delta> 0$, $\mathcal A$ has $\phi$-speed $\delta$ then so does $\cA\llcorner \sigma$.

In particular, if $\cA_1,\ldots,\cA_n$ are $\phi$-independent Alberti representations of $\mu$ then $\cA_1^*,\ldots, \cA_n^*$ are $\phi$-independent representations of $\mu^*$.
\end{lemma}

\begin{proof}
If $(\bP,\mu_\gamma)$ is in the $\phi$ direction of $C$ then $(\phi\circ\gamma)'(t) \in C$ for $\bP$-a.e. $\gamma\in \Gamma$ and $\mu_\gamma$-a.e. $t\in\dom\gamma$.  In particular it must be true that $(\phi\circ\gamma)'(t) \in C$ for $\bP$-a.e. $\gamma\in \Gamma$ and $\mu_\gamma\llcorner \gamma(\sigma(\gamma))$-a.e. $t\in \dom\gamma$.

The statement for speed is proved analogously.
\end{proof}

\begin{corollary}\label{c:restrictedspeed}
Suppose that $(X,d,\mu)$ has a universal collection of Alberti representations $\cA_1,\ldots,\cA_n$.  Then $\cA_1^*,\ldots,\cA_n^*$ is a universal collection of Alberti representations of $(X,d,\mu^*)$ (with the same constant).
\end{corollary}

In particular, by combining this with Proposition \ref{p:RNP-universal}, we see such restricted Alberti representations are RNP-universal in an RNP-LDS.

\begin{proof}
By the previous lemma, the $\cA_i^*$ are independent (because the $\cA_i$ are) and so we just need to show the speed condition.

Suppose that the $\cA_i$ are $\kappa$-universal, for some $\kappa>0$.  Then for any Lipschitz $f\colon X \to \mathbb R$ there exists a Borel decomposition $X = X_1\cup\ldots \cup X_n$ such that each $\cA_i\llcorner X_i$ has $f$-speed $\kappa$.  By the previous lemma each $\cA_i^*\llcorner X_i$ has $f$-speed $\kappa$ and so this is the required decomposition.
\end{proof}

\section{Equivalence of RNP-differentiability and connecting points}\label{s:diff}

We first define what it means for Alberti representations to connect points.

\begin{definition}\label{d:connectedalbertireps}
Let $(U,\varphi)$ be an $n$-dimensional chart in a metric measure space $(X,d,\mu)$ with $n$ $\phi$-independent Alberti representations $\cA_1,\ldots,\cA_n$ each with $\phi$-speed $\delta$.  We say that the collection $\cA_1,\ldots,\cA_n$ \emph{connects points} if there exists a $D\geq 1$ such that, for any restrictions of curves $\sigma_1,\ldots,\sigma_n$,
\begin{equation}\label{e:alberticonnects}\sup_{0<\epsilon < 1}\limsup_{X\ni x\to x_0}\frac{\tilde\rho_\epsilon(x,x_0)}{\epsilon d(x,x_0)} < D\end{equation}
for $\mu^*$-a.e. $x_0 \in U$.
\end{definition}

We note that the quantity in the above supremum increases as $\epsilon$ decreases to 0, so that the supremum is in fact a limit.

The aim of this section is to prove a relationship between RNP differentiability and Alberti representations that connect points (see Theorem \ref{t:differentiability} below).  Before doing this, we mention a few remarks regarding our definition of what it means for a collection of Alberti representations to connect points in a metric measure space.

First notice that we cannot simply ask for the \emph{existence} of curves that connect points, if we expect to describe differentiability.  Indeed, in any geodesic metric space $X$ (that is, any two points in $X$ can be connected by a \emph{curve} whose length equals the distance between the two points), for any Alberti representation, there exist curves that connect points.  However, these geodesic curves may have total $\bP$ measure zero with respect to the Alberti representation.  Thus, at the very least, we must make a statement about every $\bP$-full measure set of curve fragments.  Similar ideas show that we must also allow restricting to $\mu_\gamma$-full measure subsets of $\gamma$, for $\bP$ almost every $\gamma$.

However, this alone is not enough, and we must consider \emph{positive} measure subsets of curve fragments, not just full measure subsets, as in Definition \ref{d:connectedalbertireps}.  For a simple example, consider the two Alberti representations of $\mathcal L^2\llcorner [0,1]^2$ given by Fubini's theorem with $\bP_1$ supported on horizontal line segments and $\bP_2$ supported on vertical line segments.  One can then define the Alberti representation $\bP = (\bP_1 + \bP_2)/2$ for which every point can be connected to near by points using any $\bP$-full set of curves.  However, this one Alberti representation alone is not enough to describe the two dimensional differentiability in the square. To rule out this case, and other similar examples, we must insist upon connecting points by arbitrary restrictions of curves.

In this section we prove the following theorem.

\begin{theorem}\label{t:differentiability}
A metric measure space $(X,d,\mu)$ is a RNP Lipschitz differentiability space if and only if there exists a countable Borel decomposition $X=\cup_i U_i$ such that each $\mu\llcorner U_i$ has a collection of $n_i$ Alberti representations that connect points.  In this case, each $U_i$ is a chart together with the Lipschitz function $\phi_i\colon X\to \R^{n_i}$ that realises the independence of the Alberti representations of $\mu\llcorner U_i$, and the gradient of partial derivatives forms the derivative at almost every point.
\end{theorem}

First notice that, if $f\colon X \to V$ is differentiable at $x_0$ and a gradient $\nabla f (x_0)$ exists, then we must necessarily have $Df(x_0)=\nabla f(x_0)$ by the definition of a gradient.  If $(X,d,\mu)$ is an RNP-LDS and $V$ has the RNP, then both of these quantities exist almost everywhere and so the derivative is given by the gradient at almost everypoint.  Thus we only need to prove the implications in the theorem.  We begin with the forward implication.

\begin{proof}[Proof of forward implication of Theorem \ref{t:differentiability}]
  Let $U\subset X$ be a Borel set for which the conclusion of Theorem \ref{t:connectingpoints} holds uniformly for some $\lambda>0$.  By reducing $U$ if necessary, we may suppose that $\mu\llcorner U$ has $n$ $\kappa$-universal Alberti representations.  It suffices to prove that these Alberti representations connect points.  To this end, fix $n$ restrictions of curves, let $\mu^*$ be defined as in Definition \ref{d:abs-cont} and let $U^* = \{x \in U : d\mu^*/d\mu(x) > 0\}$.

  Let $f_{\epsilon,x}(y) := \tilde{\rho}_\epsilon(x,y)$.  A straightforward modification of the proof of Lemma \ref{l:rho-derivative} with $\rho_\epsilon$ replaced by $\tilde\rho_\epsilon$ (which shows, for each $1\leq i \leq n$, that $|(f\circ\gamma)'(t)| \leq \epsilon$ for $\mathcal L^1$-a.e. $t\in \sigma_i(\gamma)$) gives that for every $x \in U$, there is some set $U_x \subset U^*$ of full measure on which
  \begin{align*}
    \Lip(f_{\epsilon,x},y) \leq \frac{\epsilon}{\kappa}, \qquad \forall y \in U_x.
  \end{align*}
  In particular, as $f_{\epsilon,x}$ is 1-Lipschitz, the function $g_x = \epsilon\kappa^{-1} {\bf 1}_{U_x} + {\bf 1}_{U_x^c}$ is then a *-upper gradient of $f_{\epsilon,x}$.

  Let $A$ be a countable dense subset of $X$.  We set $\tilde{U} = \bigcap_{a \in A} U_a$, which is still a full measure subset of $U^*$.  Now define the function $g = \kappa^{-1} \epsilon {\bf 1}_{\tilde{U}} + {\bf 1}_{\tilde{U}^c}$.  Note this that $g$ is the same as $\max_{a \in A} g_a$.  Thus, $g$ is a *-upper gradient for $f_{\epsilon,a}$ for all $a \in A$.  We claim that
  \begin{align}
    \tilde\rho_\epsilon(x,y) \leq \int_\gamma^* g, \qquad \forall x,y \in X, \forall \gamma \in \Gamma(x,y). \label{e:delta-eps-gradient}
  \end{align}

  Assuming \eqref{e:delta-eps-gradient}, we get that
  \[\tilde\rho_\epsilon(x,y) \leq \inf_{\gamma\in \Gamma(x,y)}\int_\gamma^* g = \rho_{\epsilon/\kappa}^{\tilde{U}}(x,y).\]
  Letting $x$ be a density point of $\tilde{U}$, which has full measure in $U$, we get from Theorem \ref{t:connectingpoints} that there exists some $\lambda > 0$ so that
  \begin{align*}
    \limsup_{y \to x} \frac{\tilde\rho_\epsilon(x,y)}{\epsilon d(x,y)} \leq \limsup_{y \to x} \frac{\rho_{\epsilon/\kappa}^{\tilde{U}}(x,y)}{\epsilon d(x,y)} < \lambda.
  \end{align*}
  As this holds for arbitrary $\epsilon > 0$, this proves the theorem.

  It remains to prove \eqref{e:delta-eps-gradient}.  Let $x,y \in X$ and $\gamma\in \Gamma(x,y)$.  By translating the domain of $\gamma$ if necessary, we may suppose that $\min \dom \gamma=0$.  Now let $A \ni x_i \to x$ and for each $i\in \mathbb N$ construct the curve
  \begin{align*}
    \gamma_i : \{-d(x_i,x)\} \cup \dom\gamma &\to X \\
    t &\mapsto \begin{cases}
      x_i, & t = -d(x_i,x), \\
      \gamma(t) & t \in \dom\gamma
    \end{cases}
  \end{align*}
  which is in $\Gamma(x_i,y)$.  As $g$ is a *-upper gradient of $f_{\epsilon,x_i}$, we get
  \begin{align*}
    \tilde\rho_\epsilon(x_i,y) = |f_{\epsilon,x_i}(y) - f_{\epsilon,x_i}(x_i)| \leq \int_{\gamma_i}^* g = d(x_i,x) + \int_\gamma^* g.
  \end{align*}
  We also have that $\tilde\rho_\epsilon(x,y) \leq \tilde\rho_\epsilon(x_i,y) + d(x_i,x)$ and so
  \begin{align*}
    \tilde\rho_\epsilon(x,y) \leq 2d(x_i,x) + \int_\gamma^* g.
  \end{align*}
  As $d(x_i,x) \to 0$, we get \eqref{e:delta-eps-gradient}.
\end{proof}

Recall from \cite{bate} the notion of \emph{separated} Alberti representations:
\begin{definition}[\cite{bate}*{Definition 7.3}]
For $\xi > 0$ we say $v_1, \ldots , v_m \in R^n$ are $\xi$-separated if, for any
$\lambda \in R^m\setminus \{0\}$,
\[\left\|\sum_{i=1}^m \lambda_i v_i\right\| > \xi \max_{1\leq i \leq m}\|\lambda_i v_i\|\]
and that closed cones $C_1, \ldots , C_m$ are $\xi$-separated if any choice of $v_i \in C_i \setminus\{0\}$ are $\xi$-separated. Further, we say that Alberti representations $\cA_1, \ldots , A_m$ are $\xi$-separated if there exists $\xi$-separated cones $C_1, \ldots , C_m$ such that each $\cA_i$ in the $\phi$-direction of $C_i$.
\end{definition}
Note that independent cones (and hence Alberti representations) are $\xi$-separated for some $\xi>0$.

For the rest of this section we fix a chart $(U,\phi)$ as in the hypothesis of Theorem \ref{t:differentiability} with $\delta,D$ and $n$ fixed and suppose that the Alberti representations are $\xi$-separated for $\xi>0$ also fixed.  In particular, there exists a $\lambda>0$ depending on $\xi$ such that
\begin{equation}\label{e:phiisanorm}\limsup_{x\to x_0} \frac{|(\phi(x)-\phi(x_0))\cdot v|}{d(x,x_0)} \geq \lambda \|v\|\end{equation}
for each $v\in\mathbb R^n$.  Note that in \cite{bate}, such a $\lambda$ was guaranteed to exist because of the uniqueness of the derivative at $x_0$, rather than the existence of independent Alberti representations.

We may choose such $\lambda$ in a Borel way (see \cite{bate-speight}) and so, by passing to a subset of $U$ if necessary (this does not change the hypotheses of Theorem \ref{t:differentiability}), we may suppose that there is a fixed $\lambda>0$ such that this inequality is true for each $x_0 \in U$.  Finally, we fix a RNP Banach space $V$ and a Lipschitz function $f\colon X \to V$.

In \cite{bate} a gradient of partial derivatives of a real valued Lipschitz function $f$ was formed by selecting certain curves and taking the partial derivative of $f$ along such a curve.  Using the universal condition it was then shown that such a gradient was in fact the derivative at almost every point.  In our case, we cannot select a single curve as we must work within the framework of restrictions of curves.  However, the next definition and following lemma appear from the same motivation.

\begin{definition}
Let $p_1,\ldots,p_n \in \mathbb R^n$ be linearly independent and $f_1,\ldots,f_n\in V$.  We define $T \in L(\R^n,V)$ to be the unique linear map such that
\[f_i = T(p_i)\]
for each $1\leq i \leq n$.
\end{definition}

We will use this definition with the $p_i$ partial derivatives of $\phi$ and the $f_i$ partial derivatives of $f$, so that $T$ behaves the way a derivative should via the chain rule.  Of course, we cannot precisely guarantee that the respective partial derivatives will exactly equal $p_i$ and $f_i$ and so we need to bound $\|T\|$ to take care of the error.  The following lemma is a complete analogue to \cite{bate}*{Lemma 7.5}.

\begin{lemma}\label{l:gradientbound}
Let $p_1,\ldots,p_n\in\mathbb R^n$ and $f_1,\ldots,f_n\in V$.  Suppose that the $p_i$ are $\xi$-separated and $\|p_i\|>\delta$ and $\|f_i\|\leq L$ for each $1\leq i \leq n$.  Then
\[\|T\| \leq Ln/\xi\delta.\]
\end{lemma}

\begin{proof}
Let $v \in \mathbb S^{n-1}$ such that $\|T\| = \|T(v)\|$ and $v = \sum_i \lambda_i p_i$.  Then
\begin{align*} \|T\| &= \|T(v)\|\\
&= \left\|T\left(\sum_{i=1}^n \lambda_i p_i\right)\right\|\\
&\leq \sum_{i=1}^n |\lambda_i| \|f_i\|\\
&\leq n L \max_{1\leq i \leq n} |\lambda_i|.
\end{align*}
However, the $p_i$ are $\xi$-separated and so
\[\|v\| > \xi \max_{1\leq i \leq n}\|\lambda_i p_i\| \geq \xi \delta \max_{1\leq i \leq n}|\lambda_i|.\]
Combining these inequalities gives the required result
\end{proof}

\begin{lemma}\label{l:maindiff}
Let $0<\epsilon<1$, $f_1,\ldots,f_n \in V$ and $p_1,\ldots,p_n\in \mathbb R^n$.  Set $\sigma_1,\ldots,\sigma_n$ to be any restrictions of curves that satisfy
\[\sigma_i(\gamma)\subset \{t\in\dom\gamma : \gamma(t)\in U,\ (\phi\circ\gamma)'(t) \in B(p_i,\epsilon),\ (f\circ\gamma)'(t) \in B(f_i,\epsilon)\}.\]

Then if $x_0 \in U$ satisfies \eqref{e:alberticonnects},
\begin{equation}\label{e:approxdiff}\limsup_{X\ni x \to x_0}\frac{\|f(x)-f(x_0) - T(\phi(x)-\phi(x_0))\|}{d(x,x_0)} \leq D'\epsilon,\end{equation}
for some $D'$ that depends only on the constants fixed for this section.
\end{lemma}

\begin{proof}
First notice that it suffices to prove the lemma for sufficiently small $\epsilon$.  Therefore we may suppose that $\epsilon <\min(\Lip f, \delta\lambda/2)$ and $\epsilon$ is sufficiently small so that if $v_1,\ldots,v_n$ are $\xi$ separated with each $\|v_i\|>\delta\lambda$, then any choice of vectors $w_i\in B(v_i, \epsilon)$ for $1\leq i \leq n$ are $\xi/2$-separated.

Suppose that the Alberti representations in the hypotheses are $(\bP_i,\mu_\gamma^i)$ in the $\phi$-direction of $C_i$ for $1\leq i \leq n$ and such that $C_1,\ldots,C_n$ are $\xi$-separated.  Then for each $1\leq i \leq n$, $\bP_i$-a.e. $\gamma\in\Gamma$ and $\mu_\gamma^i$-a.e. $t\in\dom\gamma$,
\begin{align*}
(\phi\circ\gamma)'(t) &\in C_i\\
\|(\phi\circ\gamma)'(t)\| &> \delta \Lip(\phi,\gamma(t))\Lip(\gamma,t) > \delta\lambda \Lip(\gamma,t)\\
\Lip (f,\gamma(t))\Lip(\gamma,t) &\geq \|(f\circ\gamma)'(t)\|
\end{align*}
Therefore, if for each $1\leq i \leq n$ there exists a $\gamma\in\Gamma$ with $\mu_\gamma^i(\sigma_i(\gamma))>0$, we must have $\|f_i\| \leq 2\Lip f$ and $\|p_i\|>\delta\lambda/2$ for each $1\leq i \leq n$ and the $p_i$ must be $\xi/2$-separated.  If for some $1\leq i \leq n$ there does not exists such a $\gamma$  then the lemma is vacuously true (since no $x_0$ has connected neighbourhoods).  In particular, $\|T\|\leq 8\Lip f n/\xi\delta\lambda := D_1$ by Lemma \ref{l:gradientbound}.

We fix $x$ sufficiently close to $x_0$ so that $\rho_\epsilon(x,x_0)<D\epsilon d(x,x_0)$, let $\gamma\in\tilde\Gamma(x_0,x)$ with
\[\int_\gamma^* \epsilon < D\epsilon d(x,x_0)\]
and let $[a,b]$ be the smallest interval that contains $\dom\gamma$.

Note that, if $\gamma_i\in \Gamma$ and $t_i\in\R$, are as in the definition of $\gamma\in\tilde\Gamma$, then $(\phi\circ\gamma)'(t) = (\phi\circ\gamma_i)'(t-t_i)$ and $(f\circ\gamma)'(t)= (f\circ\gamma_i)'(t-t_i)$ for almost every $t\in\dom\gamma$.  Therefore the set
\[G = \{t\in\dom\gamma : \exists\ 1\leq i \leq n \text{ with } (\phi\circ\gamma)'(t) \in B(p_i,\epsilon),\ (f\circ\gamma)'(t) \in B(f_i,\epsilon)\}\]
satisfies $\mathcal L^1([a,b]\setminus G)\leq \len\gamma -\ms\gamma$.  For any $\eta>0$, by taking a suitable subset of $G$ if necessary, we may suppose that $G$ is compact and satisfies $\mathcal L^1([a,b]\setminus G)\leq \len\gamma -\ms\gamma +\eta$.

For a moment fix $t\in G$ and let $1\leq i \leq n$ be such that
\[\|(f\circ\gamma)'(t) -f_i\|<\epsilon \text{ and } \|(\phi\circ\gamma)'(t) -p_i\|<\epsilon\]
By definition, $T$ satisfies
\[f_j = T p_j \text{ for each } 1\leq j \leq n\]
and so
\begin{align*} \|(f\circ\gamma)'(t) - T ((\phi\circ\gamma)'(t))\| &\leq \|(f\circ\gamma)'(t) - f_i\| + \|f_i - T p_i\| + \|T(p_i - (\phi\circ\gamma)'(t))\|\\
&\leq \epsilon + \|T\|\epsilon\\
&\leq \epsilon(1+ 8\Lip f n/\xi\delta\lambda):=\epsilon D_2.
\end{align*}

Recall that $\gamma\in\tilde\Gamma$ is 1-Lipschitz and so $(f-T \phi)\circ\gamma$ is a $(\Lip f + D_1\Lip\phi)$-Lipschitz function defined on a subset of $[a,b]$.  We extend it (linearly on the complement of $G$) to a Lipschitz function defined on the whole of $[a,b]$ with the same Lipschitz constant and use the fundamental theorem of calculus to conclude that
\begin{align*} \|f(x)-f(x_0) -T (\phi (x)-\phi(x_0))\| &\leq\int_{G\cup([a,b]\setminus G)} \|(f\circ\gamma)'(t) - T ((\phi\circ\gamma)'(t))\|\\
&\leq \int_G \epsilon D_2 + \int_{[a,b]\setminus G} \left(\Lip f +D_1\Lip\phi\right)\\
&\leq \max(D_2,\Lip f+D_1\Lip \phi) \int_G \left(\epsilon + \len\gamma -\ms\gamma +\eta\right)\\
&\leq \max(D_2,\Lip f+D_1\Lip \phi) \int_\gamma^* \left(\epsilon +\eta\right)\\
&\leq \max(D_2,\Lip f+D_1\Lip \phi) \left(D\epsilon d(x,x_0) +\eta\right).
\end{align*}
This is true for any $\eta>0$ and so the conclusion is true for some $D'$ that depends only on $D,n,\xi,\delta,\lambda,\Lip\phi$ and $\Lip f$.
\end{proof}

Finally we prove Theorem \ref{t:differentiability}.

\begin{proof}[Proof of backward implication of Theorem \ref{t:differentiability}]
Since $f$ is Lipschitz and $X$ is separable, $f(X)$ is also separable and hence so is
\[T=\{(f\circ\gamma)'(t) \in V : \gamma\in\Gamma,\ t\in\dom \gamma,\ (f\circ\gamma)'(t) \text{ exists}\}.\]
Fix $\epsilon >0$ and let $P\subset \R^n$ and $F\subset T$ be countable such that
\[\mathbb R^n =\bigcup_{p\in P} B(p,\epsilon) \text{ and } T \subset \bigcup_{f\in F} B(f,\epsilon).\]

Let $K_m\subset U$ be an increasing sequence of compact sets such that $\mu(K_m)\to \mu(U)$ and let $Y = P^n \times F^n\times \N$.  As shown in \cite{bate}*{Lemma 2.8}, the set
\[\{(x,\gamma)\in X\times \Gamma : (f\circ\gamma)'(\gamma^{-1}(x)) \text{ and } (\phi\circ\gamma)'(\gamma^{-1}(x)) \text{ exist}\}\]
is Borel.  Therefore, for any $y = (p_1,\ldots,p_n,f_1,\ldots,f_n,m) \in Y$ we may define the restrictions of curves $\sigma_1^y,\ldots,\sigma_n^y$ by
\[\sigma_i^y(\gamma)= \{t\in\gamma^{-1}(K_m) : (\phi\circ\gamma)'(t) \in B(p_i,\epsilon),\ (f\circ\gamma)'(t) \in B(f_i,\epsilon)\}\]
 and let the induced measure be $\mu_y^*$.  Note that, by Lemma \ref{l:maindiff} and since the Alberti representations connect points, \eqref{e:approxdiff} holds for $\mu_y^*$-a.e. $x_0 \in U$ for some linear function $T_\epsilon(x_0)$ and some $D'$ independent of $x_0, y$ and $\epsilon$.

However, for a fixed $K_m$, for every $1\leq i \leq n$ and $\gamma\in\Gamma$,
\[\mathcal L^1\left(\gamma^{-1}(K_m) \setminus \bigcup_{y\in Y}\sigma_i^y\left(\gamma^{-1}(K_m)\right)\right)=0.\]
Therefore, by Lemma \ref{l:decomposerestrictions} applied to the measure $\mu\llcorner K_m$ with its induced Alberti representations, \eqref{e:approxdiff} is true for $\mu$-a.e. $x_0\in K_m$.  This is true for all $m\in \N$ and so \eqref{e:approxdiff} is true for $\mu$-a.e. $x_0\in U$.

Consider an $x_0$ in the full measure subset of $U$ that satisfies \eqref{e:approxdiff} for $\epsilon_m\to 0$ with linear function $T_m(x_0)$.  A simple triangle inequality argument shows that
\[\limsup_{x\to x_0} \frac{\|(T_m(x_0)-T_n(x_0))(\phi(x)-\phi(x_0))\|}{d(x,x_0)} \leq D'(\epsilon_n+\epsilon_m)\]
for every $n,m\in \N$.  However, a simple consequence of \eqref{e:phiisanorm} is
\[\lambda \|T_m(x_0)-T_n(x_0)\| \leq n\limsup_{x\to x_0} \frac{\|(T_m(x_0)-T_n(x_0))(\phi(x)-\phi(x_0))\|}{d(x,x_0)}.\]
Therefore, the $T_n(x_0)$ form a Cauchy sequence.  If $T(x_0)$ is its limit then the triangle inequality shows that $T(x_0)$ is actually $Df(x_0)$.
\end{proof}

\begin{remark}\label{r:connecting-char-Radon}
To conclude, we mention what happens when $(X,d,\mu)$ is simply a metric space with a Radon measure.  We can then find compact $X_i$ such that $\mu(X\setminus \cup_i X_i)=0$, each $(X_i,d)$ is complete and separable and $\mu(X_i)<\infty$. If $X$ is an RNP-LDS then so is each $X_i$ and so each $X_i$ has a decomposition as in Theorem \ref{t:differentiability}.  Since $X$ is a LDS its porous subsets all have measure zero and so Lemma \ref{l:porosity-subset} shows that the Alberti representations in fact connect almost every point to all points in $X$.

Conversely, if the Alberti representations of $X$ connect points, then the induced Alberti representations of each $X_i$ connect points and so each is an RNP-LDS.  In general this does not imply that $X$ is a RNP-LDS, see the example given in the introduction of \cite{bate-li}.  However, because the Alberti representations of $X$ connect points, every porous set in $X$ must have measure zero.

Indeed, let $S$ be a compact porous set and consider the restriction of curves that maps each $\gamma\in \Gamma$ to $\gamma^{-1}(S)$.
We consider this same restriction of curves for all Alberti representations.
For $\mu$-a.e.\ $x\in S$ and any $\epsilon>0$, if $y$ is sufficiently close to $x$, \eqref{e:alberticonnects} gives a $\gamma\in \tilde\Gamma(x,y)$ such that
\[\int_{\gamma}^{*}\epsilon \leq \epsilon D d(x,y).\]
Therefore, if $[a,b]$ is the smallest interval containing $\dom\gamma$,
\[\mathcal L^{1}([a,b]\setminus \gamma^{-1}(S)) \leq \epsilon D d(x,y).\]
In particular, since $\gamma$ is 1-Lipschitz,
\[B(y,2\epsilon D d(x,y)) \cap S \cap \gamma \neq \emptyset.\]
Thus, since $S$ is assumed to be porous, we must have $\mu(S)=0$.
In particular, Lemma \ref{l:porosity-subset} shows that $X$ is in fact a RNP-LDS.
\end{remark}

\end{document}